\documentclass[12pt,a4paper]{article}
\usepackage[a4paper,top=2.5cm,bottom=2.5cm,left=2.5cm,right=2.5cm]{geometry}
\usepackage{amsmath}
 \usepackage{graphicx}
\usepackage{amsthm}
\usepackage{color}
\usepackage{fancyhdr}
\usepackage{pdfpages}
\usepackage{txfonts}
\usepackage{amssymb}
\usepackage{paralist}
\usepackage{extarrows}
\usepackage{setspace}
\usepackage{braket} %脚注缩进的宏包
\usepackage[colorlinks=true,citecolor=blue,linkcolor=blue,anchorcolor=blue,urlcolor=blue]{hyperref}
\allowdisplaybreaks
\numberwithin{equation}{section}
\newtheorem{theorem}{Theorem}[section]
\newtheorem{lemma}{Lemma} [section]

\newtheorem{definition}[theorem]{Definition}
\newtheorem{remark}{Remark}[section]
\allowdisplaybreaks
\begin{document}
\title{Global existence and blow-up for the Hardy-Sobolev parabolic equation in $\mathbb{R}^N$
 \footnote{This study was supported by the Key Program of the National Natural Science Foundation of China (Grant No. 12231016.) and
 and the Program of the National Natural Science Foundation of China (Grant No. 12571117.)}}
\author{Fei Fang  \\  \footnotesize  \emph{School of Mathematics and Statistics,%
Beijing Technology and Business University, Beijing 100048, China}
\\Zhong Tan\footnote{Corresponding author.  Email:  tanloy2026@126.com} \\  \footnotesize  \emph{School of Mathematical Science, Xiamen University, Xiamen, 361005, China}}
\maketitle
\noindent \textbf{\textbf{Abstract:}}
In this paper, we apply a self-similar transformation to convert the parabolic equation with a Hardy term
\begin{equation*}
  \begin{cases}u_t-\Delta u-\mu \frac{u}{|x|^2}=|u|^{2^*-2} u & \text { in } \mathbb{R}^N \times(0, T), \\
 u(x, 0)=u_0(x) & \text { in } \mathbb{R}^N ,
 \end{cases}
\end{equation*}
into the following parabolic equation
\begin{equation*}
 \begin{cases}
v_s-\Delta v-\frac{1}{2} y \cdot \nabla v=\beta v+\frac{\mu v}{|y|^2}+|v|^{2^*-2} v &\text { in } \mathbb{R}^N \times(0, S), \\
 \left.v\right|_{s=0}=v_0 & \text { in } \mathbb{R}^N,
\end{cases}
\end{equation*}
where $N \geqslant 3$, $\mu\in [0,(N-2)^2 /8]$ and $2^{\ast}=2N /(N-2)$.
For this equation, we establish a weighted Hardy inequality. Furthermore, by virtue of the modified potential well method and Palais-Smale sequence analysis, we investigate the long-time behavior and finite-time blow-up properties of solutions to the parabolic equation.



\section{Introduction}
This paper  is concerned with  the following classical semilinear parabolic equation associated with critical Sobolev exponent and hardy term in $\mathbb{R}^N$ :
\begin{equation}\label{e01}
  \begin{cases}u_t-\Delta u-\mu \frac{u}{|x|^2}=|u|^{p-1} u & \text { in } \mathbb{R}^N \times(0, T), \\
 u(x, 0)=u_0(x) & \text { in } \mathbb{R}^N ,
 \end{cases}
\end{equation}
where $N \geqslant 3$, $\mu\in [0,(N-2)^2 /8]$, $p=2^{\ast}-1$ and $2^{\ast}=2N /(N-2)$.

Equation (\ref{e01}) can serve as a fundamental model for a wide range of problems originating from quantum mechanics, chemistry, cosmology, astrophysics, and differential geometry.
The inverse square potential $ 1/|x|^2$ is prevalent across multiple disciplines (see \cite{cr3,cr4}).
For instance, it emerges in point dipole interactions within molecular physics \cite{cr5}: specifically, the interaction between the electron's charge and the molecule's dipole moment gives rise to long-range forces, which in turn introduces an inverse square potential into the Schr\"{o}dinger equation describing the electron's wave function.
Furthermore, Equation (\ref{e01}) is closely intertwined with the Yamabe problem on the sphere $\mathbb{S}^N$. In fact, via stereographic projection, we can identify $\mathbb{R}^N$ with $\mathbb{S}^N$ and equip $\mathbb{S}^N$ with a metric whose scalar curvature is singular at the north pole.
The problem of finding a conformal metric with a prescribed scalar curvature of 1 thereby reduces to solving Equation (\ref{e01}) (see \cite{cr6}).

When $\mu=0$, problem \eqref{e01} becomes the following classical heat equation:
\begin{equation}\label{e001}
  \begin{cases}
  u_t-\Delta u=|u|^{q-1} u & \text { in } \mathbb{R}^N \times(0, T), \\
  u(x, 0)=u_0(x) & \text { in } \mathbb{R}^N.
  \end{cases}
\end{equation}
This equation has been studied in the literature \cite{r3,r9,r8, r5,r6,r7,tan,r4}.

 In 1966, Fujita \cite{r3} studied problem \eqref{e001} and obtained the classical critical Fujita exponent $q^{\star}=1+\frac{2}{n}$.
It is proved that when $1<q \leq q^{\star}$,
all solutions corresponding to any nonnegative nontrivial initial data cannot exist globally and must blow up in finite time.
When $q>q^{\star}$, there exists sufficiently small nonnegative initial data such that the equation admits globally bounded smooth solutions, and only solutions with large initial data blow up,
which completely characterizes the dichotomous critical behavior generated by the competition between diffusion effect and nonlinear growth effect.

 In \cite{r4}, by using the heat semigroup and contraction fixed point theorem as the main tools, Weissler studied the solvability of problem \eqref{e001} with $L^p$ initial data and obtained the Fujita exponent $q_c=1+\frac{2 p}{n}$.
Similar results to those in \cite{r3} were proved.
Thus, the classical Fujita critical blow-up theory is extended from lower-order spaces to the general $L^p$ framework, and the critical division theory of local well-posedness for semilinear parabolic equations is improved.

In \cite{r5}, a sufficient condition on the decay order of initial data, which may change sign, such that the solution
of (\ref{e001}) blows up in finite time, was given. Using self-similar transformation, Mizoguchi and Yanagida
\cite{r6,r7} established the global existence and blow-up results for problem (\ref{e001}) in $\mathbb{R}^1$.
In \cite{r9,r8}, the decay and blow-up of the solution with low energy initial data were studied by means of the potential well and forward self-similar transformation.

The following heat equation with potential has also been studied in some literature \cite{r11,r12,r1,r2,r10}:
\begin{equation}\label{e002}
 \begin{cases}
 \partial_t u-\Delta u=V(x) u+b(x) |u|^{q-1} u & \text { in } \mathbb{R}^N \times(0, T), \\
 u(x, 0)=u_0(x) & \text { in } \mathbb{R}^N,
 \end{cases}
\end{equation}
In \cite{r2}, the authors assumed that $V(x)=1$, and $0 \leq b(x) \leqslant C|x|^{m}$, $(m \in(-2, \infty))$ and proved a critical Fujita exponent $q_1$ such that
when $1<q<q_1$, there is no global solution for equation \eqref{e002}; if $q>q_1$, then both global and non-global solutions exist.

In \cite{r1}, the authors assumed that $V(x) \sim \frac{\omega}{|x|^2}$, $ b(x) \approx c_2|x|^{m}, (m \in(-\infty, \infty))$ as $|x| \rightarrow \infty$,
and the authors proved that there exists a critical Fujita exponent $q_2$; global solutions exist when $q>q_2$; while all solutions blow up in finite time when $1<q \leqslant q_2$.
In \cite{r10}, Zhang assumed that $N \geq 3$, $V(x)=\frac{a}{1+|x|^b}(b>0)$, and combined with several restrictions on parameters $a, b$, the corresponding Fujita critical exponent was obtained.
Afterwards, in \cite{r11}, Ishige relaxed the dimension condition to $N \geq 2$, and assumed that the parameter $a>0$, $V(x)=\frac{a}{|x|^2}$ and $b=1$, and also obtained the sharp Fujita critical exponent.
One can also refer to the latest research results of Ishige and Kawakami \cite{r12}.

For a general scope of this topic, we refer the interested readers to the monograph \cite{r13} and references therein.

In this article, we consider the problem (\ref{e01}) with low initial energy, critical initial energy and high initial energy. The results in our paper will be obtained by the self-similar transformation and the modified potential well method. Potential well method, which was first put forward to consider semi-linear hyperbolic initial boundary value problem by Payne and Sattinger \cite{dr10,dr11}
 around 1970s, is a powerful tool in studying the long time behaviors of solutions of some evolution equations. The potential well is defined by the level set of energy functional and the derivative functional. It is generally true that solutions starting inside the well are global in time, solutions starting outside the well and at an unstable point blow up in finite time. After the pioneer work of Sattinger and Payne,
some authors\cite{dr3,dr5,cr8,dr6,dr7,dr8,dr9,dr12,dr4}
 used the method to study the global existence and nonexistence of solutions for various nonlinear evolution equations with initial boundary value problem.
 In \cite{cr8,dr6}, Liu et al. modified and improved the method by introducing a family of potential wells which include the known potential well as a special case. The modified potential well method has been used to study semilinear pseudo-parabolic equations \cite{dr4}
and fourth-order parabolic equation \cite{dr1}. In this paper, we use the modified potential well method to obtain global existence and blow up in finite time of solutions when the initial energy is low, critical and high, respectively. When the initial energy is low, similar results are obtained in \cite{r9}, but our result is more general, moreover, we prove a more precise decay rate of the $L^2$ norm of global solution.

This paper is organized as follows. In Section 2, we give some notations, definitions and lemmas concerning the basic properties of the related functionals and sets. Sections 3 and 4 will be devoted to the cases $E_K\left(v_0\right)<d$ and $E_K\left(v_0\right)=d$, respectively, where $E_K(v)$ will be introduced in Section 2. In Section 5, we consider the case when the initial energy is high, i.e. $E_K\left(v_0\right)>d$. In section 6,  we consider the asymptotic behavior of the global solution, which is similar to
the Palais-Smale  sequence of stationary equation.

\section{Preliminaries and main lemmas}
In this section, we shall introduce the self-similar transformation and the modified potential
well method and give a series of their properties for problem (\ref{e01}). The self-similar transformation
is defined as follow:
$$
v(y, s)=(1+t)^{1 /(p-1)} u(x, t), \quad t=e^s-1, \quad x=(1+t)^{1 / 2} y.
$$
Through this transformation and the detailed calculation (see  Appendix), problem \eqref{e01} can be transformed into the following problem
\begin{equation}\label{e02}
 \begin{cases}
v_s-\Delta v-\frac{1}{2} y \cdot \nabla v=\beta v+\frac{\mu v}{|y|^2}+|v|^{p-1} v &\text { in } \mathbb{R}^N \times(0, S), \\
 \left.v\right|_{s=0}=v_0 & \text { in } \mathbb{R}^N,
\end{cases}
\end{equation}
where $S=\log (1+T)$, $\beta=(N-2)/4$. Set
$$
K(y):=e^{|y|^2 / 4}.
$$
Then we let $$L v:=-\Delta v-\frac{1}{2} y \cdot \nabla v=-\frac{1}{K} \nabla \cdot(K \nabla v).$$
We  denote by $X$ the Hilbert space obtained as the completion of $C_c^{\infty}\left(\mathbb{R}^N\right)$ with respect to the norm
$$\|\nabla v\|_{K,2}:=\left(\int_{\mathbb{R}^N}  |\nabla v|^2 K(y) dy\right)^{\frac{1}{2}}$$
which is induced by the inner product
$$
(v_1, v_2)_K:=\int_{\mathbb{R}^N} (\nabla v_1 \cdot \nabla v_2)K(y) dy.
$$
For each $q \in\left[2,2^{\ast}\right]$,  we  define $L_K^q\left(\mathbb{R}^N\right)$  and  $L_K^q\left(\mathbb{R}^N, \frac{1}{|y|^2}\right)$ as follows:
$$
L_K^q\left(\mathbb{R}^N\right):=\left\{v \text { measurable in } \mathbb{R}^N:\|v\|_{K,q}:=\left(\int_{\mathbb{R}^N}|v|^q  K(y)dy\right)^{1 / q}<\infty\right\},
$$
$$
L_K^q\left(\mathbb{R}^N, \frac{1}{|y|^2}\right):=\left\{v \text { measurable in } \mathbb{R}^N:\left\|v\right\|_{K,q,s}:=\left(\int_{\mathbb{R}^N} \frac{1}{|y|^2}|v|^q K(y)dy\right)^{1 / q}<\infty\right\} .
$$

The domain $D(L)$ of this operator consists of all $v \in L_K^2\left(\mathbb{R}^N\right)$ such that $L v \in L_K^2\left(\mathbb{R}^N\right)$,
 and one has $D(L)=X$; see Lemma 2.1 in \cite{cr1}. Moreover, $L$ is positive and self-adjoint with a compact inverse.
 In particular, the normalized eigenfunctions of $L$ constitute a complete orthonormal basis in $L_K^2\left(\mathbb{R}^N\right)$.
 The first eigenvalue of $L$ satisfies $\lambda_1=N / 2$, which yields the following Poincar\'{e} inequality
 \begin{equation}\label{po}
  \lambda_1\|v\|_{K, 2}^2 \leqslant\|\nabla v\|_{K, 2}^2, \quad v \in X,
\end{equation}
as stated in Proposition 2.3 of \cite{cr2}.

For $v \in X$, we set
$$
\begin{aligned}
E_{K}(v)&=\frac{1}{2} \int_{ \mathbb{R}^N }\left(|\nabla v|^2-\mu \frac{|v|^2}{|y|^2}
-\beta|v|^2\right) K(y)d y-\frac{1}{2^*} \int_{\mathbb{R}^N }|v|^{2^*} K(y)dy)\\
&:=\frac{1}{2}A(v)-\frac{1}{2^{\ast}}B(v),\\
D_{K}(v)&=A(v)-B(v).
\end{aligned}$$
The Nehari manifold is defined by
$$
\mathcal{N}=\left\{v \in X: D_{K}(v)=0, v \neq 0\right\},
$$
which can be separated into the two unbounded sets
$$
\begin{aligned}
& \mathcal{N}_{+}=\left\{v\in X: D_{K}(v)>0\right\}, \\
& \mathcal{N}_{-}=\left\{v \in X: D_{K}(v)<0\right\}.
\end{aligned}
$$
The potential well and its corresponding set are defined, respectively, as
$$
\begin{aligned}
W & =\left\{v \in X: D_K(v)>0, E_K(v)<d\right\} \cup\{0\}, \\
V & =\left\{v \in X: D_K(v)<0, E_K(v)<d\right\},
\end{aligned}
$$
where
$$
d=\min _{v \in X \backslash\{0\}} \max _{s \geqslant 0} E_K(s v)=\inf _{v \in \mathcal{N}} E_K(v)
$$
is the depth of the potential well $W$.

Now let us define the level set
$$
E^\alpha=\left\{v \in X: E_{K}(v)<\alpha\right\} .
$$
Furthermore, by the definition of $E_K(v), \mathcal{N}, E^\alpha$ and $d$, we easily know that
$$
\mathcal{N}^\alpha=\mathcal{N} \cap E^\alpha \equiv\left\{v \in \mathcal{N}:
A(v)<\frac{2 \alpha(p+1)}{p-1}\right\} \neq \varnothing \quad \text { for all } \alpha>d.
$$
We now define
\begin{equation*}
\lambda_\alpha=\inf \left\{\|v\|_{K,2}: v\in \mathcal{N}_\alpha\right\}, \Lambda_\alpha=\sup \left\{\|v\|_{K,2}: v \in \mathcal{N}_\alpha\right\} \text { for all } \alpha>d .
\end{equation*}
It is clear that $\lambda_\alpha$ is nonincreasing and $\Lambda_\alpha$ is nondecreasing with respect to $\alpha$.

Let us define the modified functional and Nehari manifold as follows:
\begin{equation}\begin{aligned}
  D_{K,\delta}(v)&=\delta A(v)-B(v),\\
\mathcal{N}_\delta & =\left\{u \in X: D_{K, \delta}(v)=0,\|\nabla v\|_{K,2} \neq 0\right\}, \\
d_\delta & =\inf _{v \in N_\delta} E_K(v).
\end{aligned}\end{equation}
%$$
%\mathcal{N}_\alpha=\mathcal{N} \cap E^\alpha \equiv\left\{v \in \mathcal{N}:
%\|\nabla v\|_{K, 2}^2-\mu\left\|\frac{v}{y}\right\|_{K, 2}^2-\beta\|v\|_{K, 2}^2<\sqrt{\frac{2 \alpha(p+1)}{p-1}}\right\} \neq \varnothing \quad \text { for all } \alpha>d.
%$$
Based on Lemma \ref{l02}, we see that  $A(v)>0$. Therefore, we can define
\begin{equation}\label{e03}
 S_K=\inf \left\{\left.\frac{A(v)}{\|v\|_{K, p+1}^2} \right\rvert\, v \in X\setminus\{0\}\right\}
\end{equation}
and
$$r(\delta) =\delta^{\frac{N-2}{2}} S_K^{\frac{N}{2}}.$$
 We also introduce the following sets:
$$
\begin{aligned}
\mathcal{B} & =\left\{v_0 \in X: \text { the solution } v=v(s) \text { of (\ref{e02}) blows up in finite time }\right\}, \\
\mathcal{G} & =\left\{v_0 \in X: \text { the solution } v=v(s) \text { of (\ref{e02}) exists for all } s>0\right\}, \\
\mathcal{G}_o & =\left\{v_0 \in \mathcal{G}: v(s) \mapsto 0 \text { in } X \text { as } s \rightarrow \infty\right\} .
\end{aligned}
$$
%$$
%S_K=\inf \left\{\left.\frac{\|\nabla v\|_{K, 2}^2-\mu\left\|\frac{v}{y}\right\|_{K, 2}^2
%-\beta\|v\|_{K, 2}^2}{\|v\|_{K, p+1}^2} \right\rvert\, v \in X\setminus\{0\}\right\}.
%$$
Then we can define the modified potential wells and their corresponding sets as follows:
$$
\begin{aligned}
W_\delta & =\left\{v \in X: D_{K,\delta}(v)>0, E_K(v)<d(\delta)\right\} \cup\{0\}, \\
V_\delta & =\left\{v \in X: D_{K,\delta}(v)<0, E_K(v)<d(\delta)\right\}, \\
B_\delta & =\left\{v \in X: \sqrt{A(v)}<r(\delta)\right\}, \\
B_\delta^c & =\left\{v \in X:\sqrt{A(v)}>r(\delta)\right\} .
\end{aligned}
$$

\begin{definition}\label{wd}
(Weak solution). We say that a function $v=v(y, s)$ is a weak solution of problem (\ref{e02}) in $\mathbb{R}^N_T:=\mathbb{R}^N \times(0, T)$ if and only if
$$
v \in L^{\infty}\left(0, S ; X\right), v_s \in L^2\left(\mathbb{R}^N_T\right)=L^2\left(0, S ; L_K^2(\mathbb{R}^N)\right),
$$
and satisfies problem (\ref{e02}) in the distribution sense, that is
$$\int_0^s\left(v_s, w\right)_K+\left(\nabla v, \nabla w\right)_K d\tau=\int_0^s\left(|v|^{p-1} v+\mu\frac{v}{|y|^2}+\frac{v}{p-1},
w\right)_Kd\tau, \text { for all } w \in X, s>0,$$
where $v(y, 0)=v_0(y) \in X$.
\end{definition}

For future convenience, we give some useful lemmas which will play an important role in the proof of our main results.
We first present the following Sobolev embedding theorem with the weighted function $K$.

\begin{lemma}\label{le0-1}
(See \cite{cr2}) The embedding $X \hookrightarrow L_K^q\left(\mathbb{R}^N\right)$ is continuous for all $q \in\left[2,2^{\ast}\right]$ and it is compact for all $q \in\left[2,2^{\ast}\right)$.
\end{lemma}

\begin{lemma}\label{l00}
 Assume  $u \in X$. Then we have
\begin{description}
  \item[(1)] $u /|x| \in L^2_{K}\left(\mathbb{R}^N\right)$.
  \item[(2)] (The weighted Hardy inequality)
\begin{equation}\label{wh}
  \int_{\mathbb{R}^N} \frac{|u|^2}{|x|^2} K(x)d x \leqslant C_{N} \int_{\mathbb{R}^N}|\nabla u|^2 K(x)d x
\end{equation}
with $C_{N}=\left(\frac{2}{N-2}\right)^2 $.
  \item[(3)] The constant $C_{N}$ is optimal.

\end{description}
\end{lemma}
\begin{proof}
  By a density argument, it suffices to consider only smooth functions $u \in C_0^{\infty}\left(\mathbb{R}^N\right)$. Then we have
  $$
\begin{aligned}
  |u(x)|^2K(x)&=-\int_1^{\infty} \frac{d}{d \lambda}\left(|u(\lambda x)|^2  K(\lambda x)\right)d \lambda\\
 & =-2 \int_1^{\infty} u(\lambda x)\langle x, \nabla u(\lambda x)\rangle  K(\lambda x) d \lambda-\int_1^{\infty} \left(\frac{\lambda |x|^2}{2}|u(\lambda x)|^2  K(\lambda x)\right)d \lambda.
 \end{aligned}
$$

 By using H\"{o}lder inequality,  we obtain

  $$
\begin{aligned}
\int_{\mathbb{R}^N} \frac{|u(x)|^2}{|x|^2} K(x)d x & =-2 \int_1^{\infty} \int_{\mathbb{R}^N}
\frac{u(\lambda x)}{|x|}\left\langle\frac{x}{|x|}, \nabla u(\lambda x)\right\rangle K(\lambda x)d x d \lambda \\
&\ \ \ \ -\int_1^{\infty}\int_{\mathbb{R}^N} \left(\frac{\lambda }{2}|u(\lambda x)|^2  K(\lambda x)\right)d \lambda.\\
& \leq -2 \int_1^{\infty} \frac{d \lambda}{\lambda^{N-1}} \int_{\mathbb{R}^N} \frac{u(y)}{|y|} \frac{\partial u(y)}{\partial r} K\left(y\right)d y \\
& =-\frac{2}{N-2} \int_{\mathbb{R}^N} \frac{u(y)}{|y|} \frac{\partial u(y)}{\partial r} K\left(y\right) d y \\
& \leqslant \frac{2}{N-2}\left(\int_{\mathbb{R}^N} \frac{|u(y)|^2}{|y|^2}K\left(y\right) d y\right)^{1 / 2}\left(\int_{\mathbb{R}^N}\left|\frac{\partial u(y)}{\partial r}\right|^2 K\left(y\right)d y\right)^{1 / 2}.
\end{aligned}
$$
And then we have
$$
\int_{\mathbb{R}^N} \frac{|u(x)|^2}{|x|^2} K(x) d x\leq
\left(\frac{2}{N-2}\right)^2 \int_{\mathbb{R}^N}|\nabla u(x)|^2  K(x) d x.
$$

Next, we prove that $C_N$ is the optimal constant.
Given $\varepsilon>0$, we take the following radial function
$$
V(r)= \begin{cases}1 & \text { if } r \in[0,1], \\  r^{-\gamma} e^{-\frac{|r|^2}{4}} & \text { if } r>1,\end{cases}
$$
where   $\gamma=\frac{N-2+2 \varepsilon}{2}$. Hence, we have %$B_{N,  \varepsilon}=2 /(N-2+2 \varepsilon)$,
$$
V^2(x)= \begin{cases}1 & \text { if } r \in[0,1], \\ r^{-2\gamma} e^{-\frac{|r|^2}{2}} & \text { if } r>1.\end{cases}
$$
For $r>1$, we obtain
\begin{equation*}%\label{hav}
 \begin{aligned}
V^{\prime}(r)&=-\gamma r^{-\gamma-1} e^{-\frac{r^2}{4}}+r^{-\gamma} e^{-\frac{r^2}{4}} \cdot\left(-\frac{r}{2}\right)=-e^{-\frac{r^2}{4}} r^{-\gamma}\left(\frac{\gamma}{r}+\frac{r}{2}\right) =-V(r)\left(\frac{\gamma}{r}+\frac{r}{2}\right),\\
|\nabla V(x)|^2&=\left(V^{\prime}(r)\right)^2=V^2(x)\left(\frac{\gamma^2}{|x|^2}+\gamma+\frac{|x|^2}{4}\right),\\
\frac{V^2(x)}{|x|^2}&=\frac{|\nabla V(x)|^2}{|x|^2\left(\frac{\gamma^2}{|x|^2}+\gamma+\frac{|x|^2}{4}\right)}
=\frac{|\nabla V(x)|^2}{\left(\gamma^2+\gamma|x|^2+\frac{|x|^4}{4}\right)}\geq \frac{|\nabla V(x)|^2}{\gamma^2}.
\end{aligned}
\end{equation*}
Thus, from the above relation, we get
\begin{equation}\label{hav1}
 \begin{aligned}
\int_{\mathbb{R}^N} \frac{V^2(x)}{|x|^2}K(x) d x & =\int_B \frac{V^2(x)}{|x|^2} K(x)d x+\int_{\mathbb{R}^N\backslash B} \frac{V^2(x)}{|x|^2} K(x)d x \\
& \geq\omega_N \int_0^1 r^{N-3} K(r)d r+ \frac{1}{\gamma^2}\int_{\mathbb{R}^N}|\nabla V(x)|^2K(x) d x,
\end{aligned}
\end{equation}
where $\omega_N$ is the measure of the $(N-1)$-dimensional unit sphere. %Letting $\varepsilon \rightarrow 0$, we may therefore conclude this result.

The inequality \eqref{hav1} implies that as long as the constant $1/\gamma^2$ is less than $C_N$,
we can always construct a function $V(x)$ such that the inequality \eqref{wh} fails to hold. Consequently, $C_N$ is shown to be optimal.

\end{proof}

\begin{lemma}\label{l01}
(\cite[Corollary 4.20]{cr2}) It holds that
$$
S_0\|v\|_{K, p+1}^2+\lambda_*\|v\|_{K, 2}^2 \leqslant\|\nabla v\|_{K, 2}^2, \quad v \in X,
$$
where $\lambda_*=\max (1, N / 4)$ and $S_0$ stands for the Sobolev constant:
$$
S_0=\inf \left\{\|\nabla v\|_2^2 \mid v \in C_0^{\infty}\left(\mathbb{R}^N\right),\|v\|_{p+1}=1\right\} .
$$
\end{lemma}

\begin{lemma}\label{l02}
  $S_K\geq \frac{1}{2}S_0>0$.
\end{lemma}
\begin{proof}
By Lemma \ref{l00} and Lemma \ref{l01}, we have

\begin{align*}
\left.\frac{A(v)}{\|v\|_{K, p+1}^2} \right.&=\frac{\int_{ \mathbb{R}^N }\left(|\nabla v|^2-\mu \frac{|v|^2}{|y|^2}-\beta|v|^2\right) K(y)d y }{\|v\|_{K, p+1}^2}\\
&=\frac{\int_{ \mathbb{R}^N }\left(\frac{1}{2}|\nabla v|^2-\mu \frac{|v|^2}{|y|^2}+\frac{1}{2}|\nabla v|^2-\beta|v|^2\right) K(y)d y }{\|v\|_{K, p+1}^2}\\
&\geq\frac{\frac{1}{2}\int_{ \mathbb{R}^N }\left(|\nabla v|^2-\frac{N-2}{8}|v|^2\right) K(y)d y }{\|v\|_{K, p+1}^2}\\
&\geq\frac{\frac{1}{2}\int_{ \mathbb{R}^N }\left(|\nabla v|^2-\lambda^{\ast}|v|^2\right) K(y)d y }{\|v\|_{K, p+1}^2}\geq \frac{1}{2}S_0.
\end{align*}

\end{proof}

\begin{lemma}\label{l022}
  For $v\in X$, we have $A(v)\geq \frac{1}{4} \int_{\mathbb{R}^N}|v|^2 K(y) d y$.
\end{lemma}
\begin{proof}
By Lemma \ref{l00} and  \eqref{po}, we have

$$
\begin{aligned}
A(v)&=\int_{\mathbb{R}^N}\left(|\nabla v|^2-\mu \frac{|v|^2}{|y|^2}-\beta|v|^2\right) K(y) d y\\
&=\int_{\mathbb{R}^N}\left(\frac{1}{2}|\nabla v|^2-\mu \frac{|v|^2}{|y|^2}+\frac{1}{2}|\nabla v|^2-\beta|v|^2\right) K(y) d y\\
&\geq \frac{1}{2} \int_{\mathbb{R}^N}\left(|\nabla v|^2-\frac{N-2}{4}|v|^2\right) K(y) d y\\
&\geq   \frac{1}{2} \int_{\mathbb{R}^N}\left(|\nabla v|^2-\frac{N-2}{4}|v|^2\right) K(y) d y\\
&\geq   \frac{1}{4} \int_{\mathbb{R}^N}|v|^2 K(y) d y.
\end{aligned}
$$
\end{proof}

\begin{lemma}\label{l03}Let $v \in X$. We obtain
  \begin{description}
    \item[(1)] If $0<\sqrt{A(v)}<r(\delta)$, then $D_{K, \delta}(v)>0$. In particular, if $0<\sqrt{A(v)}<r(1)$, then $D_K(v)>0$;
    \item[(2)] If $D_{K, \delta}(v)<0$, then $\sqrt{A(v)}>r(\delta)$. In particular, if $D_K(v)<0$, then $\sqrt{A(v)}>r(1) ;$
    \item[(3)] If $D_{K, \delta}(v)=0$, then $\sqrt{A(v)} \geqslant r(\delta)$ or $A(v)=0$.
    In particular, if $D_K(v)=0$, then $\sqrt{A(v)} \geqslant r(1)$ or $A(v)=0$;
    \item[(4)] If $D_{K, \delta}(v)=0$ and $A(v) \neq 0$, then $E_K(v)>0$ for $0<\delta<\frac{p+1}{2}, E_K(v)=0$ for $\delta=\frac{p+1}{2}, E_{k}(v)<0$ for $\delta>\frac{p+1}{2}$.
  \end{description}
\end{lemma}
\begin{proof}
 (1) Since $0<A(v)<r(\delta)$, by the Lemma \ref{l02} and \eqref{e03}, we have from the assumption
  $0<\sqrt{A(v)}<r(\delta):=\delta^{\frac{N-2}{2}} S_K^{\frac{N}{2}}$, and we obtain
$$\begin{aligned}
D_{K, \delta}(v) & =\delta A(v)-B(v) \\
& \geq \delta A(v)-\left(\frac{A(v)}{S_K}\right)^{\frac{p+1}{2}} \\
& \geq A(v)\left(\delta-\left(\frac{1}{S_K}\right)^{\frac{p+1}{2}}\left(A(v)\right)^{\frac{2}{N-2}}\right)>0.
\end{aligned}$$

(2) By the assumption $D_{K, \delta}(v)<0$ and \eqref{e03}, we have
 $$\begin{aligned}
0> D_{K, \delta}(v) & =\delta A(v)-B(v) \\
& \geq \delta A(v)-\left(\frac{A(v)}{S_K}\right)^{\frac{p+1}{2}} \\
& \geq A(v)\left(\delta-\left(\frac{1}{S_K}\right)^{\frac{p+1}{2}}\left(A(v)\right)^{\frac{2}{N-2}}\right).
\end{aligned}$$
Hence, $\sqrt{A(v)}>r(\delta)$.

(3) By the assumption $D_{K, \delta}(v)=0$ and \eqref{e03}, we have

 $$\begin{aligned}
0= D_{K, \delta}(v) & =\delta A(v)-B(v) \\
& \geq \delta A(v)-\left(\frac{A(v)}{S_K}\right)^{\frac{p+1}{2}} \\
& \geq A(v)\left(\delta-\left(\frac{1}{S_K}\right)^{\frac{p+1}{2}}\left(A(v)\right)^{\frac{2}{N-2}}\right).
\end{aligned}$$
Hence, $\sqrt{A(v)}\geq r(\delta)$ or $v=0$.

(4) We easily know that

$$\begin{aligned}
 E_K(v) & =\frac{1}{2}A(v)-\frac{1}{p+1}B(v) \\
 & =\left(\frac{1}{2}-\frac{\delta}{p+1}\right)A(v)+\frac{1}{p+1} (\delta A(v)-B(v))\\
 & =\left(\frac{1}{2}-\frac{\delta}{p+1}\right)A(v)+\frac{1}{p+1} D_{K, \delta}(v) \\
 & =\left(\frac{1}{2}-\frac{\delta}{p+1}\right)A(v) .
\end{aligned}$$
Then we can prove the conclusion.
\end{proof}
\begin{lemma}\label{l04}
\begin{description}
  \item[(1)]  $d(\delta) \geqslant a(\delta) r^2(\delta)$ for $a(\delta)=\frac{1}{2}-\frac{\delta}{p+1}, 0<\delta<\frac{p+1}{2}$;
  \item[(2)] $\lim _{\delta \rightarrow 0} d(\delta)=0, d\left(\frac{p+1}{2}\right)=0$ and $d(\delta)<0$ for $\delta>\frac{p+1}{2}$;
  \item[(3)] $d(\delta)$ is increasing on $0<\delta \leqslant 1$, decreasing on $1 \leqslant \delta \leqslant \frac{p+1}{2}$ and takes the maximum $d=d(1)$ at $\delta=1$.
\end{description}
\end{lemma}
\begin{proof}
(1) If $v \in \mathcal{N}_\delta$, by Lemma \ref{l03} (3), then $\sqrt{A(v)} \geqslant r(\delta)$. Moreover, we can deduce
\begin{equation} \label{e05}
\begin{aligned}
E_K(v) & =\frac{1}{2}A(v)-\frac{1}{p+1}B(v) \\
& =\left(\frac{1}{2}-\frac{\delta}{p+1}\right)A(v)+\frac{1}{p+1} D_{K, \delta}(v) \\
& =\left(\frac{1}{2}-\frac{\delta}{p+1}\right)A(v) \geq a(\delta) r^2(\delta).
\end{aligned}
\end{equation}
Hence, $d(\delta) \geqslant a(\delta) r^2(\delta)$.

(2)
 We easily know that
$$
E_K(s v)=\frac{s^2}{2}A(v)-\frac{s^{p+1}}{p+1}B(v) .
$$
Hence,
\begin{equation} \label{e04}
\lim _{s \rightarrow 0} E_K(s v)=0.
\end{equation}
And if we let $s v \in \mathcal{N}_\delta$, then $s v$ satisfies

$$
0=D_{K, \delta}(s v)=\delta s^2A(v)-s^{p+1}B(v) .
$$
Then, we obtain
\begin{equation}\label{e06}
s(\delta)=\left(\frac{\delta A(v)}{B(v)}\right)^{\frac{1}{p-1}},
\end{equation}
which yields
\begin{equation*}
\lim _{\delta \rightarrow 0} s(\delta)=0.
\end{equation*}
Now (\ref{e04}) implies that
\begin{equation*}
\lim _{\delta \rightarrow 0} E_K(s v)=\lim _{\lambda \rightarrow 0} E_K(s v)=0,
\end{equation*}
and
\begin{equation*}
\lim _{\delta \rightarrow 0} d(\delta)=0.
\end{equation*}
It is easy to see that from (\ref{e05})
$$
d\left(\frac{p+1}{2}\right)=0 \text { and } d(\delta)<0 \text { for } \delta>\frac{p+1}{2} .
$$

(3) We need to prove that for any $0<\delta^{\prime}<\delta^{\prime \prime}<1$
or $1<\delta^{\prime \prime}<\delta^{\prime}<\frac{p+1}{2}$ and for
any $w \in \mathcal{N}_{\delta^{\prime \prime}}$, there is a
 $v \in \mathcal{N}_{\delta^{\prime}}$ and a constant $\varepsilon\left(\delta^{\prime}, \delta^{\prime \prime}\right)$
 such that $E_K(v)<E_K(w)-\varepsilon\left(\delta^{\prime}, \delta^{\prime \prime}\right)$.
 Indeed, by  (\ref{e06}), we easily know that $D_{K, \delta^{\prime \prime}}(s(\delta^{\prime \prime}) w)=0$
 and $s\left(\delta^{\prime \prime}\right)=1$. Let $h(s)=E_K(s w)$
  we have
$$
\begin{aligned}
\frac{d}{d s} h(s) & =sA(w)-s^pB(w) \\
& = sA(w)-\delta^{\prime\prime} s^p A(w)+s^pD_{K,\delta^{\prime\prime}}(w)\\
& = (s-\delta^{\prime\prime} s^p) A(w)+s^pD_{K,\delta^{\prime\prime}}(w)\\
& = s(1-\delta^{\prime\prime} s^{p-1}) A(w).
\end{aligned}
$$
Take $v=s\left(\delta^{\prime}\right) w$, then $v \in \mathcal{N}_{\delta^{\prime}}$.
For $0<\delta^{\prime}<\delta^{\prime \prime}<1$,  $s\in (s(\delta^{\prime}), s(\delta^{\prime\prime}))=(s(\delta^{\prime}), 1)$, we obtain
$$
\begin{aligned}
E_K(w)-E_K(v) & =h(1)-h\left(s\left(\delta^{\prime}\right)\right) \\
& >\left(1-\delta^{\prime \prime}\right) r^2\left(\delta^{\prime \prime}\right) s\left(\delta^{\prime}\right)\left(1-s\left(\delta^{\prime}\right)\right) \equiv \varepsilon\left(\delta^{\prime}, \delta^{\prime \prime}\right).
\end{aligned}
$$
For $1<\delta^{\prime \prime}<\delta^{\prime}<\frac{p+1}{2}$, we obtain
$$
\begin{aligned}
E_K(w)-E_K(v) & =h(1)-h\left(s\left(\delta^{\prime}\right)\right) \\
& >\left(\delta^{\prime \prime}-1\right) r^2\left(\delta^{\prime \prime}\right) s\left(\delta^{\prime \prime}\right)\left(s\left(\delta^{\prime}\right)-1\right) \equiv \varepsilon\left(\delta^{\prime}, \delta^{\prime \prime}\right).
\end{aligned}
$$
Hence, the proof is complete.
\end{proof}

\begin{lemma}\label{l06}
Let $v \in X$ and $0<\delta<\frac{p+1}{2}$. If $E_K(v) \leqslant d(\delta)$, then we have
 \begin{description}
   \item[(1)] If $D_{K, \delta}(v)>0$, then $A(v)<\frac{d(\delta)}{a(\delta)}$, where $a(\delta)=\frac{1}{2}-\frac{\delta}{p+1}$. In particular, if $E_K(v) \leqslant d$ and $D_K(v)>0$, then
\begin{equation*}
A(v)<\frac{2(p+1)}{p-1} d;
\end{equation*}
   \item[(2)] If $A(v)>\frac{d(\delta)}{a(\delta)}$, then $D_{K, \delta}(v)<0$. In particular, if $E_K(v) \leqslant d$ and
\begin{equation*}
A(v)>\frac{2(p+1)}{p-1} d;
\end{equation*}
then $D_K(v)<0$.
   \item[(3)] If $D_{K, \delta}(v)=0$, then $A(v) \leqslant \frac{d(\delta)}{a(\delta)}$. In particular, if $E_K(v) \leqslant d$ and $D_K(v)=0$, then
\begin{equation*}
A(v) \leqslant \frac{2(p+1)}{p-1} d.
\end{equation*}
\end{description}
\end{lemma}
\begin{proof}
(1) For $0<\delta<\frac{p+1}{2}$, we see that
\begin{equation}\label{e07}
\begin{aligned}
E_K(v) & =\frac{1}{2}A(v)-\frac{1}{p+1}B(v) \\
& =\left(\frac{1}{2}-\frac{\delta}{p+1}\right)A(v)+\frac{1}{p+1} D_{K,\delta}(v) \\
& =a(\delta)A(v)\leq d(\delta).
\end{aligned}
\end{equation}
Therefore,
$$
A(v)<\frac{d(\delta)}{a(\delta)} .
$$
Finally, (2) and (3) follow from (\ref{e07}).
\end{proof}

\begin{lemma}\label{l07}
Let $v \in X$. We have
\begin{description}
  \item[(1)] 0 is away from both $\mathcal{N}$ and $\mathcal{N}_{-}$,
  i.e. $\operatorname{dist}(0, \mathcal{N})>0, \operatorname{dist}\left(0, N_{-}\right)>0$;
  \item[(2)]  For any $\alpha>0$, the set $E^\alpha \cap \mathcal{N}_{+}$ is bounded in $X$.
\end{description}
\end{lemma}
\begin{proof}
 (1) If $v \in \mathcal{N}$, then we have
$$
\begin{aligned}
d & \leq E_K(v)=\frac{1}{2}A(v)-\frac{1}{p+1}B(v) \\
& =\left(\frac{1}{2}-\frac{1}{p+1}\right)A(v)+\frac{1}{p+1} D_K(v) \\
& =\left(\frac{1}{2}-\frac{1}{p+1}\right)A(v).
\end{aligned}
$$
If $v \in \mathcal{N}_{-}$, then we have
$$
\begin{aligned}
d & \leq E_K(v)=\frac{1}{2}A(v)-\frac{1}{p+1}B(v) \\
& =\left(\frac{1}{2}-\frac{1}{p+1}\right)A(v)+\frac{1}{p+1} D_K(v) \\
& \leq\left(\frac{1}{2}-\frac{1}{p+1}\right)A(v).
\end{aligned}
$$
Hence, 0 is away from both $\mathcal{N}$ and $\mathcal{N}_{-}$, i.e. $\operatorname{dist}(0, \mathcal{N})>0$, $\operatorname{dist}\left(0, N_{-}\right)>0$.

(2) Since $E_K(v)<\alpha$ and $D_K(v)>0$, we obtain

$$
\begin{aligned}
\alpha & >E_K(v)=\frac{1}{2}A(v)-\frac{1}{p+1}B(v) \\
& =\left(\frac{1}{2}-\frac{1}{p+1}\right)A(v)+\frac{1}{p+1} D_K(v) \\
& >\left(\frac{1}{2}-\frac{1}{p+1}\right)A(v).
\end{aligned}
$$
Hence, for any $\alpha>0$, the set $E^\alpha \cap \mathcal{N}_{+}$is bounded in $X$.
\end{proof}

\section{Low initial energy $E_K\left(v_0\right)<d$}

The goal of this section is to prove Theorems \ref{t31}-\ref{t34}. A threshold result for the global solutions and finite time blowup will be given.

Let $A=L-\beta -\frac{\mu }{|y|^2}$, $\mu\leq\frac{(N-2)^2}{8}$. Then by means of the
 Hille-Yosida-Phillips theorem, the linear operator generates an analytic semigroup of
bounded linear operators  $\left\{e^{-t A}\right\}_{t \geq 0}$.
\textbf{Then  this structure guarantees the well-posedness of problem \eqref{e02}
locally in time.}

\begin{theorem}\label{t31}
   Assume that $v_0 \in X, S$ is the maximal existence time of $v$, and $0<e<d, \delta_1<\delta_2$ are two roots of equation $d(\delta)=e$. We have
 \begin{description}
   \item[(1)]  If $D_K\left(v_0\right)>0$, all weak solutions $v$ of problem (\ref{e02}) with $E_K\left(v_0\right)=e$ belong to $W_\delta$ for $\delta_1<\delta<\delta_2, 0 \leqslant s<S$;

   \item[(2)] If $D_K\left(v_0\right)<0$, all weak solutions $v$ of problem (\ref{e02}) with $E_K\left(v_0\right)=e$ belong to $V_\delta$ for $\delta_1< \delta<\delta_2, 0 \leqslant s<S$.
 \end{description}
\end{theorem}

\begin{theorem}\label{t32}
 (Global existence). Assume that $v_0 \in X, E_K\left(v_0\right)<d, D_K\left(v_0\right)>0$. Then problem (\ref{e02})  has a global solution $v(s) \in L^{\infty}\left(0, \infty ; X\right)$ and $v(s) \in W$ for $0 \leqslant s<\infty$.
\end{theorem}

\begin{theorem}\label{t33}
Assume that $v_0 \in X, E_K\left(v_0\right)<d$ and $D_K\left(v_0\right)<0$. Then the weak solution $v(s)$ of problem (\ref{e02})  blows up in finite time, that is, there exists a $S>0$ such that
$$
\lim _{s \rightarrow S} \int_0^s\|v(\tau)\|_{K,2} d \tau=+\infty.
$$
\end{theorem}

\begin{theorem}\label{t34}
 Assume that $v_0 \in X, E_K\left(v_0\right)<d$ and $D_K\left(v_0\right)>0, \delta_1<\delta_2$ are the two roots of equation $d(\delta)=E_K\left(v_0\right)$. Then, for the global weak solution $v$ of problem (\ref{e02}), it holds
\begin{equation}
\|v\|_{K,2}^2 \leqslant\left\|v_0\right\|_{K,2}^2 e^{-\frac{1}{2}\left(1-\delta_1\right) s}, \quad 0 \leqslant s<\infty .
\end{equation}
\end{theorem}
In order to prove Theorems \ref{t31}-\ref{t34}, we need the following lemmas:
\begin{lemma}\label{l31}
 For $0<S \leq \infty$, assume that $v: \mathbb{R}^N \times[0, S) \rightarrow \mathbb{R}^N$ is a weak solution to problem (\ref{e02}). Then it holds
\begin{equation}\label{e08}
\int_{s_1}^{s_2}\left\|v_{\tau}\right\|_{ K,2}^2 d\tau+E_K\left(v\left(s_2\right)\right)=E_K\left(v\left(s_1\right)\right), \quad \forall s_1, s_2 \in(0, S).
\end{equation}
\end{lemma}
\begin{proof}
 Multiplying (\ref{e02}) by $v_s$ and integrating over $\mathbb{R}^N$ via the integration by parts, we get (\ref{e08}).
\end{proof}

\begin{lemma}\label{l32}
 If $0<E_K(v)<d$ for some $v \in X$, and $\delta_1<1<\delta_2$ are the two roots of equation $d(\delta)=E_K(v)$, then the sign of $D_{K, \delta}(v)$ does not change for $\delta_1<\delta<\delta_2$.
\end{lemma}

\begin{proof}
Since $E_K(v)>0$, we have $\|\nabla v\|_{K,2} \neq 0$. If the sign of $D_{K, \delta}(v)$ is changeable for $\delta_1<\delta<\delta_2$, then we choose $\bar{\delta} \in\left(\delta_1, \delta_2\right)$ such that $D_{K, \bar{\delta}}(v)=0$. Hence, by the definition of $d(\bar{\delta})$, we can obtain $E_K(v) \geqslant d(\bar{\delta})$, which contradicts $E_K(v)=d\left(\delta_1\right)=d\left(\delta_2\right)<d(\bar{\delta})$ (by Lemma \ref{l04} (3)).
\end{proof}
\begin{definition}\label{df1}
   (Maximal existence time). Assume that $v(s)$ is a weak solution of problem (\ref{e02}). The maximal existence time $S$ of $v(s)$ is defined as follows:
   \begin{description}
     \item[(1)] If $v(s)$ exists for $0 \leqslant s<\infty$, then $S=+\infty$.
     \item[(2)]  If there is a $s_0 \in(0, \infty)$ such that $v(s)$ exists for $0 \leqslant s<s_0$, but doesn't exist at $s=s_0$, then $S=s_0$.

   \end{description}
\end{definition}
\begin{proof}[Proof of Theorem \ref{t31}]
  (1) Let $v(s)$ be any weak solution of problem (\ref{e02}) with $E_K\left(v_0\right)=e$, $D_K\left(v_0\right)>0$, and $S$ be the maximal existence time of $v(s)$. Using $E_K\left(v_0\right)=e, D_K\left(v_0\right)>0$ and Lemma \ref{l32},
  we have $D_{K, \delta}\left(v_0\right)>0$ and $E_K\left(v_0\right)<d(\delta)$. So $v_0(x) \in W_\delta$ for $\delta_1<\delta<\delta_2$. We need to prove that $v(s) \in W_\delta$ for $\delta_1<\delta<\delta_2$ and $0<s<S$. Indeed, if this is not the conclusion, from time continuity of $D_K(v)$ we assume that there must exist a $\delta_0 \in\left(\delta_1, \delta_2\right)$ and $s_0 \in(0, S)$ such that $v\left(s_0\right) \in \partial W_{\delta_0}$,
  and $D_{\delta_0}\left(v\left(s_0\right)\right)=0,\left\|\nabla v\left(s_0\right)\right\|_{K,2} \neq 0$ or $E_K\left(v\left(s_0\right)\right)=d\left(\delta_0\right)$. From the energy equality
\begin{equation}\label{e09}
\int_0^s \left\|v_\tau\right\|_{ K,2}^2 d\tau+E_K(v(s))=E_K\left(v_0\right)<d(\delta), \delta_1<\delta<\delta_2, \quad 0 \leqslant s<S,
\end{equation}
we easily know that $E_K\left(v\left(s_0\right)\right) \neq d\left(\delta_0\right)$. If $D_{K, \delta_0}\left(v\left(s_0\right)\right)=0,
\left\|\nabla v\left(s_0\right)\right\|_{K,2} \neq 0$, then by the definition of $d(\delta)$ we obtain $E_K\left(v\left(s_0\right)\right) \geqslant d\left(\delta_0\right)$, which contradicts (\ref{e09}).

(2) Let $v(s)$ be any weak solution of problem (\ref{e02}) with $E_K\left(v_0\right)=e, D_K\left(v_0\right)<0$,
and $S$ be the maximal existence time of $v(s)$. Using $E_K\left(v_0\right)=e, D_{K}\left(v_0\right)<0$ and
Lemma \ref{l32}, we have $D_{K,\delta}\left(v_0\right)<0$ and $E_K\left(v_0\right)<d(\delta)$. So $v_0 \in V_\delta$
for $\delta_1<\delta<\delta_2$. We need to prove that $v(s) \in V_\delta$ for $\delta_1<\delta<\delta_2$ and $0<s<S$.
 Indeed, if this is not the conclusion, from time continuity of $D_{K,\delta}(v)$ we assume that there must exist a
 $\delta_0 \in\left(\delta_1, \delta_2\right)$ and $s_0 \in(0, S)$
 such that $v\left(s_0\right) \in \partial V_{\delta_0}$, and $D_{K, \delta_0}\left(v\left(s_0\right)\right)=0$,
 or $E_K\left(v\left(s_0\right)\right)=d\left(\delta_0\right)$. From the energy equality (\ref{e09}),
 we easily know that $E_K\left(v\left(s_0\right)\right) \neq d\left(\delta_0\right)$. If $D_{K, \delta_0}\left(v\left(s_0\right)\right)=0$,
 and $s_0$ is the first time such that $D_{K, \delta_0}(v(s))=0$, then $D_{K, \delta_0}(v(s))<0$ for $0 \leqslant s<S$. By Lemma (\ref{l03}) (2),
  we have $\sqrt{A(v\left(s_0\right))}>r\left(\delta_0\right)$ for $0 \leqslant s<S$.
  So, $\sqrt{A(v\left(s_0\right))}>r\left(\delta_0\right)$ and $E_K\left(v\left(s_0\right)\right) \geq  d\left(\delta_0\right)$, which contradicts (\ref{e09}). As required.
\end{proof}

\begin{proof}[Proof of Theorem \ref{t32}]
 From  the Hille-Yosida-
Phillips theorem and the  analytic semigroup theory  of
bounded linear operator, we can prove the local existence result of (\ref{e02})
in a more general case of initial value $v_0 \in X$ and $v \in C^0\left(\left[0, S_0\right], X\right)$.

Using $E_K\left(v_0\right)<d, D_K\left(v_0\right)>0$ and Lemma \ref{l32}, we have $D_{K,\delta}\left(v_0\right)>0$ and $E_K\left(v_0\right)<d(\delta)$. So $v_0 \in W_\delta$ for $\delta_1<\delta<\delta_2$. We need to prove that $v(s) \in W_\delta$ for $\delta_1<\delta<\delta_2$ and $0<s<S$. Indeed, if this is not the conclusion, from time continuity of $D_{K,\delta}(v)$ we assume that there must exist a $\delta_0 \in\left(\delta_1, \delta_2\right)$ and $s_0 \in(0, S)$ such that $v\left(s_0\right) \in \partial W_{\delta_0}$, and $D_{K, \delta_0}\left(v\left(s_0\right)\right)=0,\left\|\nabla v(s_0)\right\|_{K,2} \neq 0$  or $E_K\left(v\left(s_0\right)\right)=d\left(\delta_0\right)$. From the energy equality
\begin{equation}\label{e10}
\int_0^s \left\|v_\tau\right\|_{K,2}^2 d\tau+E_K(v(s))=E_K\left(v_0\right)<d(\delta), \delta_1<\delta<\delta_2, \quad 0 \leqslant s<S,
\end{equation}
we easily know that $E_K\left(v\left(s_0\right)\right) \neq d\left(\delta_0\right)$.
If $D_{K, \delta_0}\left(v\left(s_0\right)\right)=0,\left\|\nabla v\left(s_0\right)\right\|_{K,2} \neq 0$,  then by the definition of $d(\delta)$ we obtain $E_K\left(v\left(s_0\right)\right) \geqslant d\left(\delta_0\right)$, which contradicts (\ref{e10}).
\end{proof}
\begin{remark}
If in Theorem \ref{t32} the condition $D_{K,\delta_2}\left(v_0\right)>0$ is replaced by $A(v_0)<r\left(\delta_2\right)$, then problem (\ref{e02}) has a global weak solution $v(s) \in L^{\infty}\left(0, \infty ; X\right)$ with $v_s(s) \in L_K^2\left(0, \infty ; X\right)$ and the following result holds
\begin{gather*}
A(v)<\frac{d(\delta)}{a(\delta)}, \quad \delta_1<\delta<\delta_2, 0 \leqslant s<\infty,  \\
\int_0^s\left\|v_\tau\right\|_{K,2}^2 d \tau<d(\delta), \quad \delta_1<\delta<\delta_2, 0 \leqslant s<\infty.
\end{gather*}
In particular
\begin{gather*}
A(v)<\frac{d\left(\delta_1\right)}{a\left(\delta_1\right)}, \\
\int_0^s\left\|v_\tau\right\|_{K,2}^2 d \tau<d\left(\delta_1\right), \quad 0 \leqslant s<\infty.
\end{gather*}

\end{remark}

\begin{proof}[Proof of Theorem \ref{t33}]
 We argue by contradiction. Suppose that there exists a global weak solution $v(s)$. Set
\begin{equation*}
f(s)=\int_0^s\|v\|_{K,2}^2 \mathrm{~d} \tau, s>0.
\end{equation*}
Multiplying (\ref{e02}) by $v$ and integrating over $\mathbb{R}^N \times(0, s)$, we get
\begin{equation*}
\int_0^s(v,v_{\tau})_{K}d\tau=\|v(s)\|_{K,2}^2-\left\|v_0\right\|_{K,2}^2=-2 \int_0^s\left(A(v)-B(v)\right)d\tau.
\end{equation*}
According to the definition of $f(s)$, we have $f^{\prime}(s)=\|v(s)\|_{K,2}^2$ and hence

\begin{equation}\label{e105}
\begin{aligned}
f^{\prime}(s)&=\|v(s)\|_{ K,2}^2=\left\|v_0\right\|_{K,2}^2-2 \int_0^s\left(A(v)-B(v)\right)d\tau\\
&=\int_0^s(v,v_{\tau})_{K}d\tau+\left\|v_0\right\|_{K,2}^2
\end{aligned}
\end{equation}
and
\begin{equation}\label{e11}
f^{\prime \prime}(s)=-2\left(A(v)-B(v)\right)=-2 D_K(v).
\end{equation}
Then by \eqref{e105}, we deduce
\begin{equation}\label{e106}
\begin{aligned}
\left[f^{\prime}(s)\right]^2&=\left[\int_0^s(v,v_{\tau})_{K}d\tau+\left\|v_0\right\|_{K,2}^2\right]^2\\
&=\left[\int_0^s(v,v_{\tau})_{K}d\tau\right]^2+2\int_0^s(v,v_{\tau})_{K}d\tau\left\|v_0\right\|_{K,2}^2+\left\|v_0\right\|_{K,2}^4\\
&=\left[\int_0^s(v,v_{\tau})_{K}d\tau\right]^2+2\left[\int_0^s(v,v_{\tau})_{K}d\tau+\left\|v_0\right\|_{K,2}^2\right]\left\|v_0\right\|_{K,2}^2 -\left\|v_0\right\|_{K,2}^4\\
&=\left[\int_0^s(v,v_{\tau})_{K}d\tau\right]^2+2f^{\prime}(s)\left\|v_0\right\|_{K,2}^2 -\left\|v_0\right\|_{K,2}^4.
\end{aligned}
\end{equation}

Now using (\ref{e10}), (\ref{e11}), Lemma \ref{l022} and
$$
\begin{aligned}
E_K(v) & =\frac{1}{2}A(v)-\frac{1}{p+1}B(v) \\
& =\frac{p-1}{2(p+1)}A(v)+\frac{1}{p+1} D_K(v),
\end{aligned}
$$
we can obtain
$$
\begin{aligned}
f^{\prime \prime}(s) & \geq 2(p+1) \int_0^s\left\|v_\tau(\tau)\right\|_{K,2}^2 \mathrm{~d} \tau+ \frac{1}{4}(p-1)\|v\|_{K,2}^2-2(p+1) E_K\left(v_0\right) \\
& =2(p+1) \int_0^s\left\|v_\tau(\tau)\right\|_{K,2}^2 \mathrm{~d} \tau+\frac{1}{4}(p-1) f^{\prime}(s)-2(p+1) E_K\left(v_0\right).
\end{aligned}
$$
Note that
$$
\begin{aligned}
f(s) f^{\prime \prime}(s)= & f(s)\left[2(p+1) \int_0^s\left\|v_\tau(\tau)\right\|_{K,2}^2 \mathrm{~d} \tau+\frac{1}{4}(p-1) f^{\prime}(s)-2(p+1) E_K\left(v_0\right)\right] \\
= & 2(p+1) \int_0^s\|v\|_{K,2}^2 \mathrm{~d} \tau \int_0^s\left\|v_\tau(\tau)\right\|_{K,2}^2 \mathrm{~d} \tau+\frac{1}{4}(p-1) f(s) f^{\prime}(s) \\
& -2(p+1) E_K\left(v_0\right) \int_0^s\|v(\tau)\|_{K,2}^2 \mathrm{~d} \tau.
\end{aligned}
$$
Hence, from \eqref{e106} we have
\begin{equation}\label{e108}
\begin{aligned}
f(s) f^{\prime \prime}(s)-\frac{p+1}{2}\left(f^{\prime}(s)\right)^2\geq & 2(p+1) \int_0^s\|v\|_{K,2}^2 \mathrm{~d} \tau \int_0^s\left\|v_\tau(\tau)\right\|_{K,2}^2 \mathrm{~d} \tau \\
& -2(p+1) \left[\int_0^s\left(v_\tau, v\right)_K \mathrm{~d} \tau\right]^2+\frac{1}{4}(p-1) f(s) f^{\prime}(s) \\
& -2(p+1) E_K\left(v_0\right) \int_0^s\|v(\tau)\|_{K,2}^2 \mathrm{~d} \tau-(p+1) f^{\prime}(s)\left\|v_0\right\|_{K,2}^2.
\end{aligned}
\end{equation}
Making use of the Schwartz inequality, we have
\begin{equation}\label{e107}
 \int_0^s \left(v, v_\tau\right)_Kd\tau \leq \int_0^s\|v\|_K\left\|v_\tau\right\|_Kd\tau
\leq \left[\int_0^s\|v\|_K^2d\tau\right]^\frac{1}{2} \left[\int_0^s\left\|v_\tau\right\|_K^2d\tau\right]^\frac{1}{2}.
\end{equation}
By combining (\ref{e108}) and (\ref{e107}), we obtain that
$$
\begin{aligned}
f(s) f^{\prime \prime}(s)-\frac{p+1}{2}\left(f^{\prime}(s)\right)^2 \geq & \frac{1}{4}(p-1) f(s) f^{\prime}(s)-(p+1) f^{\prime}(s)\left\|v_0\right\|_{K,2}^2 \\
& -2(p+1) E_K\left(v_0\right) \int_0^s\|v(\tau)\|_{K,2}^2 \mathrm{~d} \tau.
\end{aligned}
$$
Next, we distinguish two cases:

(1) If $E_K\left(v_0\right) \leqslant 0$, then
\begin{equation*}
f(s) f^{\prime \prime}(s)-\frac{p+1}{2}\left(f^{\prime}(s)\right)^2 \geq \frac{1}{4}(p-1) f(s) f^{\prime}(s)-(p+1) f^{\prime}(s)\left\|v_0\right\|_{K,2}^2.
\end{equation*}
Now we prove
\begin{equation}\label{sue}
D_K(v)<0 \ \text{for}\  s>0.
\end{equation}
If not, we must be allowed to choose a $s_0>0$
such that $D_K\left(v\left(s_0\right)\right)=0$ and $D_K(v)<0$ for $0 \leqslant s<s_0$.
From Lemma \ref{l03} (2), we have  $\sqrt{A(v)}>r(1)$ for $0 \leqslant s<s_0, \sqrt{A(v(s_0))} \geqslant r(1)$
and $E_K\left(v\left(s_0\right)\right) \geqslant d$, which contradicts (\ref{e08}).
From (\ref{e11}) we have $f^{\prime}(s)>0$ for $s \geqslant 0$. From $f^{\prime}(0)=\|v(0)\|_{K,2}^2 \geqslant 0$,
we can know that there exists a $s_0 \geqslant 0$ such that $f^{\prime}\left(s_0\right)>0$. For $s \geqslant s_0$ we have
\begin{equation*}
f(s) \geqslant f^{\prime}\left(s_0\right)\left(s-s_0\right)>f^{\prime}(0)\left(s-s_0\right).
\end{equation*}
Hence, for sufficiently large $s$, we obtain
\begin{equation*}
f(s)>(p+1)\left\|v_0\right\|_{K,2}^2,
\end{equation*}
then
$$
f(s) f^{\prime \prime}(s)-\frac{p+1}{2}\left(f^{\prime}(s)\right)^2>0.
$$

(2) If $0<E_K\left(v_0\right)<d$, then by Theorem \ref{t31} we have $v(s) \in V_\delta$ for $1<\delta<\delta_2, s \geqslant 0$,
and $D_\delta(v)<0,\sqrt{A(v)}>r(\delta)$ for $1<\delta<\delta_{2, s} \geqslant 0$, where $\delta_2$ is the larger root of equation $d(\delta)=E_K\left(v_0\right)$. Hence, $D_{\delta_2}(v) \leqslant 0$ and $\sqrt{A(v)} \geqslant r\left(\delta_2\right)$ for $s \geqslant 0$. By (\ref{e11}), we have

$$
\begin{aligned}
f^{\prime \prime}(s) & =-2 D_K(v)=2\left(\delta_2-1\right)\left(A(v)\right)-2 D_{\delta_2}(v), \\
& \geqslant 2\left(\delta_2-1\right)\left(A(v)\right) \geqslant 2\left(\delta_2-1\right) r^2\left(\delta_2\right), \quad s \geqslant 0,  \\
f^{\prime}(s) & \geqslant 2\left(\delta_2-1\right) r^2\left(\delta_2\right) s+f^{\prime}(0) \geqslant 2\left(\delta_2-1\right) r^2\left(\delta_2\right) s, \quad s \geqslant 0, \\
f(s) & \geqslant\left(\delta_2-1\right) r^2\left(\delta_2\right) s^2, \quad s \geqslant 0 .
\end{aligned}
$$
Therefore, for sufficiently large $s$, we infer
\begin{equation*}
\frac{(p-1)}{8} f(s)>(p+1)\left\|v_0\right\|_{K,2}^2, \quad \frac{(p-1)}{8} f^{\prime}(s)>2(p+1) E_K\left(v_0\right).
\end{equation*}
Then, (\ref{e108}) and Lemma \ref{l022} imply that
$$
\begin{aligned}
f(s) f^{\prime \prime}(s)-\frac{p+1}{2}\left(f^{\prime}(s)\right)^2 \geq & \frac{1}{4}(p-1) f(s) f^{\prime}(s)-(p+1) f^{\prime}(s)\left\|v_0\right\|_{K,2}^2 \\
& -2(p+1) f(s) E_K\left(v_0\right) \int_0^s\|v(\tau)\|_{ K,2}^2 \\
= & \left(\frac{(p-1)}{8} f(s)-(p+1)\left\|v_0\right\|_{K,2}^2\right) f^{\prime}(s) \\
& +\left(\frac{(p-1)}{8} f^{\prime}(s)-2(p+1) E_K\left(v_0\right)\right) f(s)>0.
\end{aligned}
$$
The remainder of the proof is the same as that in \cite{cr8}.

\end{proof}

\begin{proof}[Proof of Theorem \ref{t34}]
  Multiplying (\ref{e02}) by $w, w \in L^{\infty}\left(0, \infty ; H_{K}^1(\mathbb{R}^N\right)$, we have
\begin{equation}\label{e12}
\left(v_s, w\right)_K+(\nabla v, \nabla w)_K=\left(|v|^{p-1} v+\mu\frac{v}{|y|^2}+\frac{v}{p-1}, w\right)_K .
\end{equation}
Letting $w=v$, (\ref{e12}) implies that
\begin{equation}\label{e13}
\frac{1}{2} \frac{d}{d s}\|v\|_{K,2}^2+D_K(v)=0, \quad 0 \leqslant s<\infty .
\end{equation}
From $0<E_K\left(v_0\right)<d, D_K\left(v_0\right)>0$ and Theorem \ref{t31}, we have $v(s) \in W_\delta$ for $\delta_1<\delta<\delta_2$ and $0 \leqslant s<\infty$, where $\delta_1<\delta_2$ are the two roots of equation $d(\delta)=E_K\left(v_0\right)$. Hence, we obtain $D_{K, \delta}(v) \geqslant 0$ for $\delta_1<\delta<\delta_2$ and $D_{K, \delta_1}(v) \geqslant 0$ for $0 \leqslant s<\infty$. So, (\ref{e13}) gives
\begin{equation}\label{e13b}
\frac{1}{2} \frac{d}{d s}\|v\|_{K,2}^2+\left(1-\delta_1\right)A(v)+D_{K, \delta_1}(v)=0, \quad 0 \leqslant s<\infty .
\end{equation}
Now (\ref{e13b}) and Lemma \ref{l022} imply that
\begin{equation*}
\frac{1}{2} \frac{d}{d s}\|v\|_{K,2}^2+\frac{1}{4}\left(1-\delta_1\right)\|v\|_{K,2}^2 \leq 0, \quad 0 \leqslant s<\infty,
\end{equation*}
and
\begin{equation*}
\|v\|_{K,2}^2 \leqslant\left\|v_0\right\|_{K,2}^2-\frac{1}{2}\left(1-\delta_1\right) \int_0^s|v(\tau)|^2 d \tau, \quad 0 \leqslant s<\infty .
\end{equation*}
By Gronwall's inequality, we have
\begin{equation*}
\|v\|_{K,2}^2 \leqslant\left\|v_0\right\|_{K,2}^2 e^{- \frac{1}{2}\left(1-\delta_1\right) s}, \quad 0 \leqslant s<\infty .
\end{equation*}
This completes the proof.
\end{proof}
\section{ Critical initial energy $E_K\left(v_0\right)=d$}
The goal of this section is to prove Theorem \ref{t41}, Theorem \ref{t42} and Theorem \ref{tt43}.
\begin{theorem}\label{t41}
 (Global existence). Assume that $v_0 \in X, E\left(v_0\right)=d$ and $D_K\left(v_0\right) \geqslant 0$. Then problem (2.1) has a global weak solution $v(s) \in L^{\infty}\left(0, \infty ; X\right)$ and $v(s) \in \bar{W}=W \cup \partial W$ for $0 \leqslant s<\infty$.
\end{theorem}

\begin{lemma}\label{l41}
Assume that $v \in X,\|\nabla v\|_{K,2}^2 \neq 0$, and $D_K(v) \geq 0$. Then
  \begin{description}
    \item[(1)]  $\lim _{\mu \rightarrow 0} E_K(\lambda v)=0, \lim _{\mu \rightarrow+\infty} E_K(\mu v)=-\infty$;
    \item[(2)]  On the interval $0<\mu<\infty$, there exists a unique $\mu^*=\mu^*(u)$, such that

\begin{equation*}
\left.\frac{d}{d \mu} E_K(\mu v)\right|_{\mu=\mu^*}=0;
\end{equation*}

    \item[(3)] $E_K(\mu v)$ is increasing on $0 \leqslant \mu \leqslant \mu^*$, decreasing on $\mu^* \leqslant \mu<\infty$ and takes the maximum at $\mu=\mu^*$;
    \item[(4)]  $D_K(\mu v)>0$ for $0<\mu<\mu^*, D_K(\mu v)<0$ for $\mu^*<\mu<\infty$, and $D_K\left(\mu^* v\right)=0$.
\end{description}
\end{lemma}

\begin{proof}
  (1) Firstly, from the definition of $E_K(v)$, i.e.
$$
E_K(v)=\frac{1}{2}A(v)-\frac{1}{p+1}B(v)
$$
and we see that
$$
E_K(\mu v)=\frac{1}{2}A(\mu v)-\frac{1}{p+1}B(\mu v).
$$
Hence, we have
\begin{equation*}
\lim _{\mu \rightarrow 0} E_K(\mu v)=0 \quad \text { and } \quad \lim _{\mu \rightarrow+\infty} E_K(\mu u)=-\infty .
\end{equation*}

(2) It is easy to show that
$$
\frac{d}{d \mu} E_K(\mu v)=\mu A(v)-\mu^p B(v),
$$
which leads to the conclusion.

(3) By Lemma \ref{l41} (2), one has
$$
\begin{array}{ll}
\frac{d}{d \mu} E_K(\mu v)>0 & \text { for } 0<\mu<\mu^*, \\
\frac{d}{d \mu} E_K(\mu v)<0 & \text { for } \mu^*<\mu<\infty,
\end{array}
$$
which leads to the conclusion.

(4) The conclusion follows from

$$
D_K(\mu v)=\frac{d}{d \mu} E_K(\mu v)=\mu A(v)-\mu^p B(v).
$$
As desired.
\end{proof}

\begin{proof}[Proof of Theorem \ref{t41}]
 Firstly, $E_K\left(v_0\right)=d$ implies that $\left\|v_0\right\|_{k,2} \neq 0$. Choose a sequence $\left\{\mu_m\right\}$ such that $0<\mu_m<1, m=1,2, \ldots$ and $\mu_m \rightarrow 1$ as $m \rightarrow \infty$. Let $v_{0 m}=\mu_m v_0$. We consider the following initial problem

\begin{equation}\label{e14}
 \begin{cases}v_s+L v=|v|^{p-1} v+\mu\frac{v}{|y|^2}+\frac{1}{p-1} v & \text { in } \mathbb{R}^N \times(0, S),  \\
 \left.v\right|_{s=0}=v_{0 m} & \text { in } \mathbb{R}^N .\end{cases}
\end{equation}
From $D_K\left(v_0\right) \geqslant 0$ and Lemma \ref{l41}, we have $\mu^*=\mu^*\left(v_0\right) \geqslant 1$. Thus, we get $D_K\left(v_{0 m}\right)= D_K\left(\mu_m v_0\right)>0$ and $E_K\left(v_{0 m}\right)=E_K\left(\mu_m v_0\right)<E_K\left(v_0\right)=d$. From Theorem \ref{t32}, it follows that for each $m$ problem (\ref{e14}) admits a global weak solution $v_m(s) \in L^{\infty}\left(0, \infty ;  X\right)$ with $v_{m }(s) \in L^2\left(0, \infty ;  X\right)$ and $v_m(s) \in W$ for $0 \leqslant s<\infty$ satisfying

\begin{gather*}
\left(v_{m,s}, w\right)_K+\left(\nabla v_{m}, \nabla w\right)_K=\left(|v_m|^{p-1} v_m+\mu\frac{v_m}{|y|^2}+\frac{v_m}{p-1}, w\right)_K, \text { for all } w \in  X, s>0,  \\
\int_0^s\left\|v_{m, \tau}\right\|_{K,2}^2+E_K\left(v_m(s)\right)=E_K\left(v_{0 m}\right)<d, \quad 0 \leqslant s<\infty,
\end{gather*}
which implies that
$$
\begin{aligned}
E_K\left(v_m\right) & =\frac{1}{2}A(v_m)-\frac{1}{p+1}B(v_m) \\
& =\frac{p-1}{2(p+1)}A(v_m)+\frac{1}{p+1} D_K\left(v_m\right).
\end{aligned}
$$
Therefore, one has
\begin{equation*}
\int_0^S\left\|v_{m, \tau}\right\|_{K,2}^2 \mathrm{~d} \tau+\frac{p-1}{2(p+1)}A(v_m)<d, \quad 0 \leqslant s<\infty .
\end{equation*}
The remainder of the proof is similar to the proof of Theorem \ref{t32}.
\end{proof}
\begin{theorem}\label{t42}
 (Blow-up). Assume that $v_0 \in X, E_K\left(v_0\right)=d$ and $D_k\left(v_0\right)<0$.
 Then the existence time of weak solution for problem (\ref{e02}) is finite.
\end{theorem}
\begin{proof}[Proof of Theorem \ref{t42}]
 Let $v(s)$ be any weak solution of problem (\ref{e02}) with $E_K\left(v_0\right)=d$ and $D_K\left(v_0\right)<0, S$ be the existence time of $v(s)$.
 We next prove $S<\infty$. We argue by contradiction. Suppose that there would exist a global weak solution $v(s)$. Set
\begin{equation*}
f(s)=\int_0^s\|v\|_{K,2}^2 \mathrm{~d} \tau, s>0.
\end{equation*}
Multiplying (\ref{e02}) by $u$ and integrating over $\mathbb{R}^N \times(0, s)$, we get

\begin{equation*}
\|v(s)\|_{K,2}^2-\left\|v_0\right\|_{K,2}^2=-2 \int_0^s\left(A(v)-B(v)\right)d\tau .
\end{equation*}
According to the definition of $f(s)$, we have $f^{\prime}(s)=\|v\|_{K,2}^2$ and hence
\begin{equation*}
\begin{aligned}
f^{\prime}(s)&=\|v(s)\|_{K,2}^2=\left\|v_0\right\|_{K,2}^2-2 \int_0^s\left(A(v)-B(v)\right)d\tau\\
&=\int_0^s(v,v_{\tau})_{K}d\tau+\left\|v_0\right\|_{K,2}^2
\end{aligned}
\end{equation*}
and
\begin{equation}\label{e15}
f^{\prime \prime}(s)=-2\left(A(v)-B(v)\right)d\tau=-2 D_K(v) .
\end{equation}
Now using (\ref{e10}), (\ref{e15}) and
$$
\begin{aligned}
E_K(v) & =\frac{1}{2}A(v)-\frac{1}{p+1}B(v)\\
& =\frac{p-1}{2(p+1)}A(v)+\frac{1}{p+1} D_K(v),
\end{aligned}
$$
we can obtain
$$
\begin{aligned}
f^{\prime \prime}(s) & \geq 2(p+1) \int_0^s\left\|v_\tau(\tau)\right\|_{K,2}^2 \mathrm{~d} \tau+ \frac{1}{4}(p-1)\|v\|_{K,2}^2-2(p+1) E_K\left(v_0\right) \\
& =2(p+1) \int_0^s\left\|v_\tau(\tau)\right\|_{K,2}^2 \mathrm{~d} \tau+ \frac{1}{4}(p-1) f^{\prime}(s)-2(p+1) E_K\left(v_0\right).
\end{aligned}
$$
Note that
$$
\begin{aligned}
f(s) f^{\prime \prime}(s)= & f(s)\left[2(p+1) \int_0^s\left\|v_\tau(\tau)\right\|_{K,2}^2 \mathrm{~d} \tau+ \frac{1}{4}(p-1) f^{\prime}(s)-2(p+1) E_K\left(v_0\right)\right] \\
= & 2(p+1) \int_0^s\|v\|_{K,2}^2 \mathrm{~d} \tau \int_0^s\left\|v_\tau(\tau)\right\|_{K,2}^2 \mathrm{~d} \tau+\frac{1}{4}(p-1) f(s) f^{\prime}(s) \\
& -2(p+1) E_K\left(v_0\right) \int_0^s\|v(\tau)\|_{K,2}^2 \mathrm{~d} \tau.
\end{aligned}
$$
Hence, we have
\begin{equation}\label{e16}
 \begin{aligned}
f(s) f^{\prime \prime}(s)-\frac{p+1}{2}\left(f^{\prime}(s)\right)^2= & 2(p+1) \int_0^s\|v\|_{K,2}^2 \mathrm{~d} \tau \int_0^s\left\|v_\tau(\tau)\right\|_{K,2}^2 \mathrm{~d} \tau \\
& -2(p+1) \left[\int_0^s\left(v_\tau, v\right)_K \mathrm{~d} \tau\right]^2+\frac{1}{4}(p-1) f(s) f^{\prime}(s) \\
& -2(p+1) E_K\left(v_0\right) \int_0^s\|v(\tau)\|_{K,2}^2 \mathrm{~d} \tau-(p+1) f^{\prime}(s)\left\|v_0\right\|_{K,2}^2.
\end{aligned}
\end{equation}
Hence, according to \eqref{e106}, \eqref{e107} and  (\ref{e16}), we obtain
\begin{equation}\label{e17}
 \begin{aligned}
f(s) f^{\prime \prime}(s)-\frac{p+1}{2}\left(f^{\prime}(s)\right)^2 \geq & \frac{1}{4}(p-1) f(s) f^{\prime}(s)-(p+1) f^{\prime}(s)\left\|v_0\right\|_{K,2}^2 \\
& -2(p+1) f(s) E_K\left(v_0\right) \int_0^s\|v(\tau)\|_{K,2}^2 \mathrm{~d} \tau \\
= & \left(\frac{(p-1)}{8} f(s)-(p+1)\left\|v_0\right\|_{K,2}^2\right) f^{\prime}(s) \\
& +\left(\frac{(p-1)}{8} f^{\prime}(s)-2(p+1) E_K\left(v_0\right)\right) f(s).
\end{aligned}
\end{equation}
On the other hand, from $E_K\left(v_0\right)=d>0, D_K\left(v_0\right)<0$ and the continuity of $E_K(v)$ and $D_K(v)$ with respect to $s$,
it follows that there exists a sufficiently small $s_1>0$ such that $E_K\left(v\left(s_1\right)\right)>0$ and $D_K(v)<0$ for $0 \leqslant s \leqslant s_1$.
 Hence $\left(v_s, v\right)_K=-D_K(v)>0,\left\|v_s\right\|_{K,2}>0$ for $0 \leqslant s \leqslant s_1$.
 So, using the continuity of $\int_0^s\left\|v_\tau\right\|_{K,2}^2 d \tau$, we can choose a $s_1$ such that
\begin{equation*}
0<d_1=d-\int_0^{s_1}\left\|v_\tau\right\|_{K,2}^2 d \tau<d.
\end{equation*}
And by (\ref{e08}), we get
\begin{equation*}
0<E_K\left(v\left(s_1\right)\right)=d-\int_0^{s_1}\left\|v_\tau\right\|_{K,2}^2 d \tau=d_1<d.
\end{equation*}
So we can choose $s=s_1$ as the initial time, then we obtain $v(s) \in V_\delta$
for $\delta \in\left(\delta_1, \delta_2\right), s_1 \leqslant s<\infty$,
 where $\left(\delta_1, \delta_2\right)$ is the maximal interval including
  $\delta=1$ such that $d(\delta)>d_1$ for $\delta \in\left(\delta_1, \delta_2\right)$.
  Thus we get $D_{K, \delta}(v)<0$ and $\sqrt{A(v)}>r(\delta)$ for $\delta \in\left(1, \delta_2\right), s_1 \leqslant s<\infty$,
  and $D_{K, \delta_2}(v) \leqslant 0, \sqrt{A(v)} \geqslant r\left(\delta_2\right)$
  for $s_1 \leqslant s<\infty$. Thus (\ref{e15}) implies that

$$
\begin{aligned}
f^{\prime \prime}(s) & =-2 D_K(v)=2\left(\delta_2-1\right)\left(A(v)\right)-2 D_{\delta_{K, 2}}(v)\geqslant 2\left(\delta_2-1\right) r\left(\delta_2\right), \quad s \geqslant s_1, \\
f^{\prime}(s) & \geqslant 2\left(\delta_2-1\right) r\left(\delta_2\right)\left(s-s_1\right)+f^{\prime}\left(s_1\right) \geqslant 2\left(\delta_2-1\right) r\left(\delta_2\right)\left(s-s_1\right), \quad s \geqslant s_1, \\
f(s) & \geqslant\left(\delta_2-1\right) r\left(\delta_2\right)\left(s-s_1\right)^2+C
>\left(\delta_2-1\right) r\left(\delta_2\right)\left(s-s_1\right)^2, \quad s \geqslant s_1 .
\end{aligned}
$$
Therefore, for sufficiently large $s$, we infer

$$
\frac{(p-1)}{8} f(s)>(p+1)\left\|v_0\right\|_{K,2}^2, \quad \frac{(p-1)}{8} f^{\prime}(s)>2(p+1) E_K\left(v_0\right) .
$$
Then, (\ref{e17}) implies that
$$
\begin{aligned}
f(s) f^{\prime \prime}(s)-\frac{p+1}{2}\left(f^{\prime}(s)\right)^2 \geq & \frac{1}{4}(p-1) f(s) f^{\prime}(s)-(p+1) f^{\prime}(s)\left\|v_0\right\|_{K,2}^2 \\
& -2(p+1) f(s) E_K\left(v_0\right) \int_0^s\|v(\tau)\|_{K,2}^2 \mathrm{~d} \tau \\
= & \left(\frac{(p-1)}{8} f(s)-(p+1)\left\|v_0\right\|_{K,2}^2\right) f^{\prime}(s) \\
& +\left(\frac{(p-1)}{8} f^{\prime}(s)-2(p+1) E_K\left(v_0\right)\right) f(s)>0.
\end{aligned}
$$
The remainder of the proof is the same as that in \cite{cr8}.
\end{proof}

\begin{theorem}\label{tt43}
 Assume that $v_0 \in X, E_K\left(v_0\right)=d$ and $D_K\left(v_0\right)>0, \delta_1<\delta_2$ are the two roots of equation $d(\delta)=E_K\left(v_0\right)$. Then, for the global weak solution $v$ of problem (\ref{e02}), it holds
\begin{equation*}
|v|_2^2 \leqslant\left|v_0\right|_2^2 e^{-\frac{1}{2}\left(1-\delta_1\right) s}, \quad 0 \leqslant s<\infty .
\end{equation*}
\end{theorem}
\begin{proof}
 We first know that problem (\ref{e02}) has a global weak solution from Theorem \ref{t41}. Furthermore,
 Using Theorem \ref{t42} and (\ref{e08}), if $v(s)$ is a global weak solution of problem (\ref{e02}) with $E_K\left(v_0\right)=d, D_K\left(v_0\right)>0$, then must have $D_K(v) \geq 0$ for $0 \leq s<+\infty$. Next, we distinguish two cases:

(1) Suppose that $D_K(v)>0$ for $0 \leqslant s<\infty$. Multiplying (\ref{e02}) by $w, w \in L^{\infty}\left(0, \infty ; X\right)$, we have

\begin{equation}\label{e18}
\left(v_s, w\right)_K+\left(\nabla v, \nabla w\right)_K=\left(|v|^{p-1} v+\mu\frac{v}{|y|^2}+\frac{v}{p-1}, w\right)_K, \text { for all } w \in X, s>0 .
\end{equation}
Letting $w=v$, (\ref{e18}) implies that
\begin{equation}\label{e19}
\frac{1}{2} \frac{d}{d s}\|v\|_{K,2}^2=-D_K(v)<0, \quad 0 \leqslant s<\infty.
\end{equation}
Since $\left\|v_s\right\|_{K,2}>0$, we have that $\int_0^s\left\|v_\tau\right\|_{K,2}^2 d \tau$
 is increasing for $0 \leqslant s<\infty$. By choosing any $s_1>0$ such that
\begin{equation}\label{e20}
0<d_1=d-\int_0^{s_1}\left\|v_\tau\right\|_{K,2}^2 d \tau<d.
\end{equation}
From (\ref{e08}), if follows that $0<E_K(v) \leq d_1<d$, and $v(s) \in W_\delta$
for $\delta_1<\delta<\delta_2$ and $0 \leqslant s<\infty$, where $\delta_1<\delta_2$ are the two roots of equation $d(\delta)=E_K\left(v_0\right)$.
Hence, we obtain $D_{K, \delta_1}(v) \geqslant 0$ for $\delta_1<\delta<\delta_2$ and $D_{K, \delta_1}(v) \geqslant 0$ for $s_1 \leqslant s<\infty$. So, (\ref{e19}) gives
$$
\frac{1}{2} \frac{d}{d s}\|v\|_{K,2}^2+\left(1-\delta_1\right)\left(A(v)\right)+D_{K, \delta_1}(v)=0, \quad 0 \leqslant s<\infty.
$$
Now Lemma \ref{l022} and (\ref{e19}) imply that
$$
\frac{1}{2} \frac{d}{d s}\|v\|_{K,2}^2+\frac{1}{4}\left(1-\delta_1\right)\|v\|_{K,2}^2 \leq 0, \quad 0 \leqslant s<\infty
$$
and
$$
\|v\|_{K,2}^2 \leqslant\left\|v_0\right\|_{K,2}^2-\frac{1}{2}\left(1-\delta_1\right) \int_0^s\|v\|_{K,2}^2 d \tau, \quad 0 \leqslant s<\infty .
$$
and by Gronwall's inequality, we have
\begin{equation*}
\|v\|_{K,2}^2 \leqslant\left\|v_0\right\|_2^2 e^{-\frac{1}{2}\left(1-\delta_1\right) s}, \quad 0 \leqslant s<\infty.
\end{equation*}

(2) Suppose that there exists a $s_1>0$ such that $D_K\left(v\left(s_1\right)\right)=0$ and $D_K(v)>0$ for $0 \leqslant s<s_1$.
 Then, $\left\|v_s\right\|_{K,2}>0$ and $\int_0^s\left\|v_\tau\right\|_{K,2}^2 d \tau$ is increasing for $0 \leqslant s<s_1$. By (\ref{e20}) we have
\begin{equation*}
E_K\left(v\left(s_1\right)\right)=d-\int_0^{s_1}\left\|v_\tau\right\|_{K,2}^2 d \tau<d
\end{equation*}
and $\left\|v\left(s_1\right)\right\|_{K,2}=0$. Then, we have that $v(s) \equiv 0$ for $s_1 \leqslant s<\infty$.
Hence, the proof is complete.

\end{proof}

\section{High initial energy $E_K\left(v_0\right)>d$}

In this section, we investigate the conditions to ensure the existence of global solutions or blow-up solutions to problem \eqref{e02} with $E_K\left(v_0\right)>d$.
\begin{lemma}\label{l51}
For any $\alpha>d, \lambda_\alpha$ and $\Lambda_\alpha$ defined in \eqref{e02}  satisfy
\begin{equation*}
0<\lambda_\alpha \leq \Lambda_\alpha<+\infty.
\end{equation*}
\end{lemma}

\begin{proof}
  (1) By H\"{o}lder's inequality, fundamental inequality and $u \in \mathcal{N}$, we have

\begin{equation*}
A(v)=B(v).
\end{equation*}
Then from Lemma \ref{l07} (1), we have $\lambda_\alpha>0$.
Using Lemma \ref{l01},  Lemma \ref{l02}  and $v \in \mathcal{N}$, we have

\begin{equation*}
A(v)=B(v) \leqslant\left(\frac{1}{S_K}\left(A(v)\right)\right)^{\frac{p+1}{2}} .
\end{equation*}
So we have $A(v) \leq \frac{1}{S_K}$ which leads to the conclusion.
\end{proof}

\begin{theorem}\label{t51}
   Suppose that $E_K\left(v_0\right)>d$, then we have
\begin{description}
  \item[(1)] If $v_0 \in \mathcal{N}_{+}$ and $\left\|v_0\right\|_{K,2} \leq \lambda_{E_K\left(v_0\right)}$, then $v_0 \in \mathcal{G}_0$;
  \item[(2)] If $v_0 \in \mathcal{N}_{-}$ and $\left\|v_0\right\|_{K,2} \geq \Lambda_{E_K\left(v_0\right)}$, then $v_0 \in \mathcal{B}$.
\end{description}
\end{theorem}

\begin{proof}
The maximal existence time of the solutions to problem (\ref{e02}) with initial value $v_0$ is denoted by $S_0$.
If the solution is global, i.e. $S\left(v_0\right)=+\infty$, the limit set of $v_0$ is denoted by $\omega_0$.

(1) Suppose that $v_0 \in \mathcal{N}_{+}$ with $\left\|v_0\right\|_{K,2} \leq \lambda_{E_K\left(v_0\right)}$.
We firstly prove that $v(s) \in \mathcal{N}_{+}$ for all $s \in \left[0, S\left(v_0\right)\right)$.
 Assume, on the contrary, that there exists a $s_0 \in\left(0, S\left(v_0\right)\right)$
 such that $v(s) \in \mathcal{N}_{+}$for $0 \leq s<s_0$ and $v\left(s_0\right) \in \mathcal{N}$.
 It follows from $D_K(v(s))=-\int_{\mathbb{R}^N} v_s(yx, s) v(y, s) K(y)\mathrm{d} y$ that $v_s(y, s) \neq 0$
 for $(y, s) \in \mathbb{R}^N \times\left(0, s_0\right)$.
 Recording to (\ref{e08}) we then have $E_K\left(v\left(s_0\right)\right)<E_K\left(v_0\right)$,
 which implies that $v\left(s_0\right) \in E_K^{E_K\left(v_0\right)}$. Therefore, $v\left(s_0\right) \in \mathcal{N}^{E_K\left(v_0\right)}$.
 Recalling the definition of $\lambda_{E_K\left(v_0\right)}$, we get

\begin{equation}\label{e21}
\left\|v\left(s_0\right)\right\|_{K,2} \geq \lambda_{E_K}\left(v_0\right).
\end{equation}
Since $D_K(v(s))>0$ for $s \in\left[0, s_0\right)$, we obtain from (\ref{e13}) that
\begin{equation*}
\left\|v\left(s_0\right)\right\|_{K,2}<\left\|v_0\right\|_{K,2} \leq \lambda_{E_K\left(v_0\right)},
\end{equation*}
which contradicts (\ref{e21}). Hence, $v(s) \in \mathcal{N}_{+}$ which shows that $v(s) \in E_K^{E_K\left(v_0\right)}$ for all $s \in \left[0, S\left(v_0\right)\right)$.
Now Lemma \ref{l07} (2) implies that the orbit $\{v(s)\}$ remains bounded in $X$
for $s \in\left[0, S\left(v_0\right)\right)$ so that $S\left(v_0\right)=\infty$. Assume that $\omega$ is an arbitrary element in $\omega\left(v_0\right)$.
Then by (\ref{e08}) and (\ref{e13}) we obtain
\begin{equation*}
\|\omega\|_{K,2}<\lambda_{E_K\left(v_0\right)}, \quad E_K(\omega)<E_K\left(v_0\right),
\end{equation*}
which, according to the definition of $\lambda_{E_K\left(v_0\right)}$ again, implies that $\omega\left(v_0\right) \cap \mathcal{N}=\varnothing$. So, $\omega\left(v_0\right)= \{0\}$, i.e. $v_0 \in \mathcal{G}_0$.

(2) Suppose that $v_0 \in \mathcal{N}_{-}$with $\left\|v_0\right\|_{K,2} \geq \Lambda_{E_K\left(v_0\right)}$.
 We now prove that $v(s) \in \mathcal{N}_{-}$ for all $s \in \left[0, S\left(v_0\right)\right)$.
 Assume, on the contrary, that there exists a $s^0 \in\left(0, S\left(v_0\right)\right)$ such that
 $v(s) \in \mathcal{N}_{-}$ for $0 \leq s<s^0$ and $v\left(s^0\right) \in \mathcal{N}$.
 Similarly to case (1), one has $E_K\left(v\left(s^0\right)\right)<E_K\left(v_0\right)$,
  which implies that $v\left(s^0\right) \in E_K^{E_K\left(v_0\right)}$. Therefore, $v\left(s^0\right) \in \mathcal{N}^{E_K\left(v_0\right)}$.
 Recalling the definition of $\Lambda_{E_K\left(v_0\right)}$, we infer

\begin{equation}\label{e22}
\left\|v\left(s^0\right)\right\|_{K,2} \leq \Lambda_{E_K\left(v_0\right)} .
\end{equation}
On the other hand, from (\ref{e13}) and the fact that $D_K(v(s))<0$ for $s \in\left[0, s^0\right)$, we obtain
\begin{equation*}
\left\|v\left(s^0\right)\right\|_{K,2}>\left\|v_0\right\|_{K,2} \geq \Lambda_{E_K\left(v_0\right)},
\end{equation*}
which contradicts (\ref{e22}).

Assume that $S\left(v_0\right)=\infty$. Then for each $\omega \in \omega\left(v_0\right)$, it follows from by (\ref{e08}) and (\ref{e13}) that

\begin{equation*}
\|\omega\|_{K,2}>\Lambda_{E_K\left(v_0\right)}, \quad E_K(\omega)<E_K\left(v_0\right).
\end{equation*}
Noting the definition of $\Lambda_{E_K\left(v_0\right)}$ again, we have $\omega\left(v_0\right) \cap N=\varnothing$. Hence, it is holded that $\omega\left(v_0\right)=\{0\}$, which contradicts Lemma \eqref{l07} (1). Therefore, $S\left(v_0\right)<\infty$. This ends the proof.

\end{proof}

\begin{theorem}\label{t52}
Assume that $v_0 \in X$ satisfies
\begin{equation}\label{e23}
  \frac{8(p+1)}{p-1}E_K\left(v_0\right) \leq\left\|v_0\right\|_{K,2}<\frac{p-1}{2(p+1)}\left\|v_0\right\|_{K,p+1}^{p+1}.
\end{equation}
Then, $v_0 \in \mathcal{N}_{-} \cap \mathcal{B}$.
\end{theorem}
\begin{proof}
 Firstly, we observe
$$
\begin{aligned}
E_K\left(v_0\right) & =\frac{1}{2}A(v_0)-\frac{1}{p+1}B(v_0) \\
& =\frac{1}{2} D_K\left(v_0\right)+\frac{p-1}{2p+2}B(v_0).
\end{aligned}
$$
Thus, \eqref{e23} implies that
$$
E_K\left(v_0\right)-\frac{p-1}{2p+2}B(v_0)=\frac{1}{2} D_K\left(v_0\right)<0,
$$
which shows that $v_0 \in \mathcal{N}_{-}$. Then for any $v \in \mathcal{N}^{E_K\left(v_0\right)}$, by Lemma \ref{l022} and  (\ref{e23}) one has
$$
\frac{1}{4}\|v\|_{K,2}\leq B(v)=A(v)  \leq \frac{2 E_K\left(v_0\right) (p+1)}{p-1}\leq \frac{1}{4}\|v_0\|_{K,2}.
$$
Taking supremum over $\mathcal{N}^{E_K\left(v_0\right)}$  and by Theorem \ref{t51} we can deduce
$$
\left\|v_0\right\|_{K,2} \geq \Lambda_{E_K\left(v_0\right)} .
$$
Thus, $v \in \mathcal{N}_{-} \cap \mathcal{B}$. This finishes the proof.
\end{proof}

\section{Stationary solutions}
In this section,  we consider the asymptotic behavior of the global solutions, which is similar to
the Palais-Smale  sequence of stationary equation.

\begin{theorem}\label{t61}
Let $v\left(y, s ; v_0\right)$ be a global solution of the problem (\ref{e02}) and uniformly bounded in $X$ with respect to $s$.
Then, for any subsequence $s_n \rightarrow \infty$,
there exists a stationary solution $w$ such that $v\left(y, s_n ; v_0\right) \rightharpoonup w$ in $X$.
\end{theorem}

\begin{theorem}\label{t62}
 Let $v\left(y, s ; v_0\right)$ be a global solution of the problem (\ref{e02}). Then, its $\omega$-limit contains a stationary solution $w$.
 The $\omega$-limit set is defined as
$$
\omega\left(v_0\right)=\left\{w \in X \mid \exists s_n \rightarrow+\infty, v\left(x, s_n ; v_0\right) \rightharpoonup w \text { in } X\right\}.
$$
\end{theorem}

\begin{proof}[Proof of Theorem \ref{t61}]
 Let us denote $v_n:=v\left(y, s_n\right)$. Since $\left\{v_n\right\}$ is uniformly bounded in $X$,
 then using Lemma \ref{le0-1} and Lemma \ref{l00}   there exists a subsequence (here we still denote by $\left\{v_n\right\}$ ) and a function $w\in  X$ such that

$$
\begin{aligned}
v_n \rightharpoonup w & \text { in } X, \\
v_n \rightarrow w & \text { in } L_{K}^2(\mathbb{R}^N), \\
v_n \rightharpoonup w & \text { in } L_{K}^2\left(\mathbb{R}^N,\frac{1}{|y^2|}\right), \\
v_n \rightarrow w & \text { a.e. in } \mathbb{R}^N.
\end{aligned}
$$
Let $U_n:=v\left(s_n+s\right)$ for $s \in(0,1)$. Clearly, $U_n$ is uniformly bounded in $X$, we show

$$
U_n \rightarrow w \text { in } L_{K}^2(\mathbb{R}^N).
$$
Indeed, for $s \in(0,1)$, by (\ref{e10}), we have
\begin{equation}
\int_0^{\infty} \left\|v_\tau\right\|_{K,2}^2 d\tau+E_K(w))=E_K\left(v_0\right)<\infty, \quad 0 \leqslant s<S,
\end{equation}
which means $v_\tau \in L_k^2(\mathbb{R}^N)$. So

$$
\int_{\mathbb{R}^N}\left|U_n-v_n\right|^2 K(y)d y\leq s \int_{s_n}^{s+s_n} \int_{\mathbb{R}^N}\left|v_\tau\right|^2 K(y)d y d \tau \rightarrow 0
$$
for $0 \leq s \leq 1$ as $s_n \rightarrow \infty$, which implies that $\left\|v\left(s+s_n\right)-v\left(s_n\right)\right\|_{K,2} \rightarrow 0$ as $s_n \rightarrow \infty$ for $0 \leq s \leq 1$. Therefore, we have

\begin{equation}\label{sh1}
 U_n \rightarrow w \text { in }  L_{K}^2(\mathbb{R}^N)
\end{equation}
and
\begin{equation}\label{sh2}
  U_n \rightarrow w \text { a.e. in } \mathbb{R}^N.
\end{equation}
Since $\left\{U_n\right\}$ is uniformly bounded in $X$, by Lemma \ref{le0-1} , we obtain
\begin{equation}\label{sh3}
 U_n \rightharpoonup w \text { in }   L_{K}^2\left(\mathbb{R}^N,\frac{1}{|y^2|}\right),
\end{equation}
and
\begin{equation}\label{sh4}
 |U_n|^{2^{\ast}-2}U_n \rightharpoonup |w|^{2^{\ast}-2}w,  \text { in } \  L_K^\frac{2 N}{N+2}\left(\mathbb{R}^N\right).
\end{equation}

In order to show that $w$ is a stationary solution, we pass to the limit (as $s_n \rightarrow \infty$ ) in  Definition \ref{wd} with a suitably chosen test function. Let

$$
\phi(x, t)= \begin{cases}\rho\left(s-s_n\right) \Psi(x), & s>s_n, y \in \mathbb{R}^N, \\
0, & 0 \leq s \leq s_n, y \in \mathbb{R}^N,\end{cases}
$$
where
$$
\Psi \in X, \rho \in C_0^2(0,1), \rho \geq 0, \int_0^1 \rho(t) d t=1.
$$
Take $\phi$ as test function in Definition \ref{wd}, we have
$$\begin{aligned}
&\int_{s_n}^{s_n+1} \int_{\mathbb{R}^N}\left[v \rho^{\prime}\left(s-s_n\right)
\Psi-\rho\left(s-s_n\right) \nabla v \nabla \Psi+\mu\frac{v}{|y|^2}\rho\left(s-s_n\right) \Psi\right.\\
&\quad\quad+\left.\frac{v}{p-1}\rho\left(s-s_n\right) \Psi+|v|^{p-1}v \rho\left(s-s_n\right) \Psi\right] K(y)\mathrm{d} y \mathrm{~d} s=0.
\end{aligned}$$
The transformation $\tau=s-s_n$ leads to
\begin{equation}\label{e61}
 \begin{aligned}
&\int_{0}^{1} \int_{\mathbb{R}^N}\left[v(\tau+s_n) \rho^{\prime}\left(\tau\right)
\Psi-\rho\left(\tau\right) \nabla v(\tau+s_n) \nabla \Psi+\mu\frac{v(\tau+s_n)}{|y|^2}\rho\left(\tau\right) \Psi\right.\\
&\quad\quad+\left.\frac{v(\tau+s_n)}{p-1}\rho\left(\tau\right) \Psi+|v(\tau+s_n)|^{p-1}v(\tau+s_n) \rho\left(\tau\right) \Psi\right] K(y)\mathrm{d} y \mathrm{~d}\tau=0.
\end{aligned}
\end{equation}
Now we rewrite Equation (\ref{e61}) as follows:

\begin{align*}
 &\int_0^1 \int_{\mathbb{R}^N}\left[v\left(s_n\right) \rho^{\prime}(\tau) \Psi-\rho(\tau) \nabla v\left(s_n\right) \nabla \Psi+|v\left(s_n\right)|^{p-1}v\left(s_n\right) \rho(\tau) \Psi\right.\\
&\quad+\left.\frac{\mu}{|y|^2}v\left(s_n\right) \rho(\tau) \Psi+\frac{1}{p-1}v\left(s_n\right) \rho(\tau) \Psi\right]K(y \mathrm{d} y \mathrm{~d} \tau \\
& \quad+\int_0^1 \int_{\mathbb{R}^N}\left[v\left(s_n+\tau\right)-v\left(s_n\right)\right] \rho^{\prime}(\tau) \Psi K(y)\mathrm{d} y \mathrm{~d} \tau \\
& \quad-\int_0^1 \int_{\mathbb{R}^N}\left[\nabla v\left(s_n+\tau\right)-\nabla v\left(s_n\right)\right] \rho(\tau) \nabla \Psi K(y) \mathrm{d} y \mathrm{~d} \tau \\
& \quad+\int_0^1 \int_{\mathbb{R}^N}\left[|v\left(s_n+\tau\right)|^{p-1}v\left(s_n+\tau\right)-|v\left(s_n\right)|^{p-1}v\left(s_n\right)\right] \rho(\tau) \Psi K(y) \mathrm{d} y\mathrm{~d} \tau\\
&\quad+\int_0^1 \int_{\mathbb{R}^N}\left[\frac{\mu}{|y|^2}v\left(s_n+\tau\right)-\frac{\mu}{|y|^2}v\left(s_n\right)\right] \rho(\tau) \Psi K(y) \mathrm{d} y\mathrm{~d} \tau\\
&\quad+\int_0^1 \int_{\mathbb{R}^N}\left[\frac{1}{p-1}v\left(s_n+\tau\right)-\frac{1}{p-1}v\left(s_n\right)\right] \rho(\tau) \Psi K(y) \mathrm{d} y\mathrm{~d} \tau=0.
\end{align*}
By the dominated convergence theorem  and  \eqref{sh1}-\eqref{sh4}%$v\left(s_n\right) \rightarrow w$ strongly in $L_K^2(\mathbb{R}^N)$,
we have
$$
\begin{aligned}
& \int_0^1 \int_{\mathbb{R}^N}\left[\rho(\tau) \nabla v\left(s_n\right) \nabla \Psi-\left|v\left(s_n\right)\right|^{p-1} v\left(s_n\right) \rho(\tau) \Psi\right. \\
& \left.\quad-\frac{\mu}{|y|^2} v\left(s_n\right) \rho(\tau) \Psi-\frac{1}{p-1} v\left(s_n\right) \rho(\tau) \Psi\right] K(y \mathrm{~d} y \mathrm{~d} \tau=
o(1), \ \text { as } n  \rightarrow \infty .
\end{aligned}
$$
From the choice of $\rho$, we obtain
$$
\begin{aligned}
& \int_0^1 \int_{\mathbb{R}^N}\left[ \nabla v\left(s_n\right) \nabla \Psi-\left|v\left(s_n\right)\right|^{p-1} v\left(s_n\right)  \Psi\right. \\
& \left.\quad-\frac{\mu}{|y|^2} v\left(s_n\right) \Psi-\frac{1}{p-1} v\left(s_n\right)  \Psi\right] K(y \mathrm{~d} y \mathrm{~d} \tau
=o(1), \ \text { as } n  \rightarrow \infty,
\end{aligned}
$$
which completes the proof of Theorem \ref{t61}.
\end{proof}

\begin{proof}[Proof of Theorem \ref{t62}]
 Let $v=v(s, y)$ be a global solution of problem (\ref{e02}). Then, we have

\begin{equation}\label{e63}
\int_0^{\infty} \int_{\mathbb{R}^N} v_s^2 K(y)d y d s \leq C<\infty.
\end{equation}
And hence, there exists a sequence $\left\{s_n\right\}$ satisfying $s_n \rightarrow \infty$ as $n \rightarrow \infty$ such that
\begin{equation}\label{e64}
\int_{\mathbb{R}^N}\left|v_s\left(s_n, y\right)\right|^2 K(y) d y \rightarrow 0, \text { as } n \rightarrow \infty.
\end{equation}
Indeed, on the contrary, if there exist $c>0$ such that $\int_{\mathbb{R}^N}\left|v_s\left(t_n, x\right)\right|^2 d y>c$ as $n \rightarrow \infty$,
then, we can derive a contradiction with (\ref{e63}).

Next, letting $v_n:=v\left(s_n, y\right)$,  we have $E_K(v(s_n))>0$ for $s_n>0$.
Indeed, on the contrary,  if $E_K(v(s_n))\leq 0$, from the proof of \eqref{sue},  we obtain $D_K(v(s_n))< 0$.
Hence, Theorem \ref{t33} implies that $v(s_n)$  blows up in finite time. This is a contradiction. So, $E_K(v(s_n))>0$.

Then, by (\ref{e10}), we easily know that

\begin{equation}\label{e65}
0<E_K\left(v\left(s_n\right)\right) \leq E_K\left(v_0\right).
\end{equation}
Then, (\ref{e64}) and (\ref{e65}) imply that $v_n:=v\left(s_n, y\right)$ is a Palais-Smale
sequence  related to the stationary equation of problem (\ref{e02}). Similar to the standard proof
of boundedness for  Palais-Smale sequence of elliptic equation
(see \cite{hr13}),
it is easy to prove that there exists a constant $C$ such that

$$
\int_{\mathbb{R^N}}\left|\nabla v_n\right|^2 K(y)d y \leq C,
$$
and then there exists a subsequence (denote still by $\left\{v_n\right\}$ ) and a function $w$ such that
\begin{equation}\label{fin}
 \begin{aligned}
&v_n \rightharpoonup w,  \text { in } X, \\
& v_n \rightarrow w,  \text { in } L_K^q(\mathbb{R}^N)\left(2 \leq q<2^*\right),\\
&v_n \rightharpoonup w,  \text { in } \  L_{K}^2\left(\mathbb{R}^N,\frac{1}{|y^2|}\right),\\
& |v_n|^{2^{\ast}-2}v_n \rightharpoonup |w|^{2^{\ast}-2}w,  \text { in } \  L_K^\frac{2 N}{N+2}\left(\mathbb{R}^N\right),
\end{aligned}
\end{equation}
Furthermore, using \eqref{e65} and \eqref{fin}, one has $v_n \rightharpoonup w \neq 0$ in $X$,  which means that $w$ is a nontrivial stationary solution.
\end{proof}

\section*{Appendix}
In this section, we employ self-similar transformations to derive (2.1).
Firstly, we can know that the transform
$$
v(y, s)=(1+t)^{1 /(p-1)} u(x, t), \quad t=e^s-1, \quad x=(1+t)^{1 / 2} y,
$$
where $p=2^{\ast}-1.$
Let $\rho=1+t, \beta=\frac{1}{p-1}$. Then  we have
$$x=\rho^{\frac{1}{2}} y,\quad   |x|^2=\rho|y|^2,\quad   t=e^s-1, \quad   \rho=e^s,\quad  \beta=\frac{N-2}{4}$$
and
$$ v(y, s)=\rho^\beta u(x, t),\quad   u(x, t)=\rho^{-\beta} v(y, s),\quad v(y, s)=v\left(\rho^{-\frac{1}{2}} x, s\right).$$
Moreover, we get
$$
\partial_t \rho=1, \quad \partial_s \rho=\rho, \quad \frac{\partial s}{\partial t}=\frac{1}{\rho},
$$
$$\frac{\partial y_j}{\partial t}=-\frac{1}{2} \rho^{-\frac{3}{2}} x_j=-\frac{1}{2} \rho^{-1} y_j,$$

$$\partial_t v=\frac{\partial v}{\partial y_j} \frac{\partial y_j}{\partial t}+\frac{\partial v}{\partial s} \frac{\partial s}{\partial t}=-\frac{1}{2 \rho} y \cdot \nabla_y v+\frac{1}{\rho} \partial_s v$$ and
$$
\begin{aligned}
 \partial_t u&=-\beta \rho^{-\beta-1} v+\rho^{-\beta} \partial_t v \\
 &=-\beta \rho^{-\beta-1} v+\rho^{-\beta}\left(-\frac{1}{2 \rho} y \cdot \nabla v+\frac{1}{\rho} v_s\right)\\
 &=\rho^{-\beta-1}\left(v_s-\frac{1}{2} y \cdot \nabla v-\beta v\right).
\end{aligned}
$$
Further calculations imply that
$$
\begin{aligned}
&\Delta_x=\sum_i \partial_{x_i}^2=\rho^{-1} \Delta_y,\\
&\Delta_x u=\rho^{-1} \Delta_y\left(\rho^{-\beta} v\right)=\rho^{-\beta-1} \Delta_y v,
\end{aligned}
$$
$$
\begin{aligned}
 \mu \frac{u}{|x|^2}=\mu \frac{\rho^{-\beta} v}{\rho|y|^2}=\rho^{-\beta-1} \frac{\mu v}{|y|^2}
\end{aligned}
$$
and
$$
\begin{aligned}
|u|^{2^*-2} u=\left|\rho^{-\beta} v\right|^{2^*-2} \rho^{-\beta} v=\rho^{-\left(2^*-2\right) \beta} \rho^{-\beta}|v|^{2^*-2} v=\rho^{-\beta-1}|v|^{2^*-2} v.
\end{aligned}
$$
Hence,  $$u_t-\Delta u-\frac{\mu u}{|x|^2}=|u|^{2^*-2} u$$ can be  transformed into

$$
\rho^{-\beta-1}\left(v_s-\frac{1}{2} y \cdot \nabla v-\beta v\right)-\rho^{-\beta-1} \Delta v-\rho^{-\beta-1} \frac{\mu v}{|y|^2}=\rho^{-\beta-1}|v|^{2^*-2} v,
$$
that is,
$$
v_s-\frac{1}{2} y \cdot \nabla v-\Delta v-\beta v-\frac{\mu v}{|y|^2}=|v|^{2^*-2} v.
$$

\end{document}